\documentclass[a4paper,twoside]{article}

\usepackage{derksen}

\author{{\bf Gregor Kemper} \\
  \normalsize Technische Universit\"at M\"unchen, Zentrum Mathematik - M11 \\
  \normalsize  Boltzmannstr. 3, 85\,748 Garching, Germany \\
   \normalsize {\tt kemper$@$ma.tum.de}}

 \title{\bf Using Extended Derksen Ideals in Computational Invariant
   Theory}

\begin{document}

\maketitle

\begin{abstract}
  \normalsize \noindent The main purpose of this paper is to develop
  new algorithms for computing invariant rings in a general
  setting. This includes invariants of nonreductive groups but also of
  groups acting on algebras over certain rings. In particular, we
  present an algorithm for computing invariants of a finite group
  acting on a finitely generated algebra over a Euclidean ring. This
  may be viewed as a first step in ``computational arithmetic
  invariant theory.'' As a special case, the algorithm can compute
  multiplicative invariant rings. Other algorithms are applicable to
  nonreductive groups and are, when applied to reductive groups, often
  faster than the algorithms known to date.

  The main tool is a generalized and modified version of an ideal that
  was already used by Derksen in his algorithm for computing
  invariants of linearly reductive groups. As a further application,
  these so-called extended Derksen ideals give rise to
  invariantization maps, which turn an arbitrary ring element into an
  invariant.

  For the most part, the algorithms of this paper have been
  implemented. 
\end{abstract}


\section*{Introduction} \label{sIntro}%

The computation of invariant rings is a classical problem in invariant
theory. It is well-known that all invariant rings of reductive groups
are finitely generated. So are invariant rings of finite groups acting
on finitely generated algebras over a Noetherian ring. To date, most
computational methods are applicable only to reductive groups acting
on affine varieties (see \mycite{Derksen:99} and
\mycite{kem.separating}) or to finite groups acting on finitely
generated algebras over a field (see \mycite{kem:e} and
\mycite[Section~2.1]{Kamke:Diss}). However, nonreductive groups are
often important in practice (such as in applications to image
processing, where nonreductivity occurs in the guise of translational
motions), and experience shows that their invariants are often rather
harmless. Therefore it is desirable to have algorithmic methods for
dealing with such invariants. It would also be desirable to be able to
compute invariant rings over rings, such as $\ZZ$, rather than over
fields. For example, multiplicative invariants (see
\mycite{Lorenz:2005}) are invariants over $\ZZ$. Moreover, if $G
\subseteq \GL_n(\mathcal{O})$ with $\mathcal{O}$ an integral extension
of $\ZZ$, then $\mathcal{O}[x_1 \upto x_n]^G$ often provides a
``universal'' invariant ring that specializes to all $K[x_1 \upto
x_n]^G$ for $K$ a field with a map $\mathcal{O} \to K$. For example,
this works if $G$ is a permutation group or, more generally, a
monomial group. So the invariant ring over $\mathcal{O}$ displays all
phenomena that occur in the various characteristics.

First steps toward computing invariants of nonreductive groups were
taken by \mycite{essen}, \mycite{Derksen.Kemper06}, and
\mycite{Kamke:Diss} (see also \mycite{Kamke:Kemper:2011}). In
particular, \mycite{Kamke:Diss} modified an algorithm by
\mycite{MQB:99} (see also \mycite{HK:06}) for computing invariant
fields and arrived at an algorithm for computing a localization of the
invariant ring of a unipotent group. (In fact, Kamke's algorithm is
applicable whenever the invariant field equals the ring of fractions
of the invariant ring.) The algorithms of \mycite{MQB:99} and
\mycite{Kamke:Diss} on the one hand, and the celebrated algorithm of
\mycite{Derksen:99} for computing invariant rings of linearly
reductive group on the other hand, are based on very different ideas,
but uncannily they both use the same ideal as a main computational
tool. This ideal has become known as the {\it Derksen ideal} (see, for
example, \mycite{kem.separating}), and it corresponds to the {\em
  graph of the action}. \mycite{HK:06} modified the algorithm of
\mycite{MQB:99} for computing invariant fields by introducing a {\em
  cross-section}, which is a subvariety that intersects with a generic
orbit in finitely many points. This can greatly increase the
efficiency of the algorithm. \citename{HK:06}'s work prompted
\mycite{Kamke:Kemper:2011} to define {\em extended Derksen ideals}, a
purely algebraic notion that captures the idea of cross-sections.

The research of this paper started as an attempt to carry those ideas
further. We give a new definition of an extended Derksen ideal, which
applies in a generalized situation. In doing so, we introduce tamely
and nontamely extended Derksen ideals, where the tame ones lead to an
algorithm for the computation of invariant fields and correspond to
\citename{HK:06}'s idea of a cross section. But nontamely extended
Derksen ideals form a larger class and sometimes enable the
computation of an invariant ring when tamely extended Derksen ideals
fail. Another way in which the concept of extended Derksen ideals from
this paper is more general is that they are defined over rings rather
than fields. In the case of finite groups, this additional generality
leads to an algorithm for computing generating sets of an invariant
ring of a finite group acting on a finitely generated domain $R$ over
a Euclidean ring or, more generally, any ring that allows Gr\"obner
basis computations. As a special case, the algorithm can compute
multiplicative invariant rings. So this paper solves one of the open
problems (Problem~7) from Lorenz' book~[\citenumber{Lorenz:2005}]. We
also give an algorithm, albeit a less efficient one, that does not
require $R$ to be an integral domain. This makes Noether's finiteness
result~[\citenumber{noe:c}] constructive in a generality that may be
impossible to extend. Further results about the computation of
invariants of infinite groups will be mentioned below. \\

\sref{sDerksenA} of the paper is devoted to the definition of extended
Derksen ideals and to the basic results pertaining to them. Apart from
allowing the computation of invariant fields and localizations of
invariant rings, they also give rise to {\em invariantization} maps,
i.e., maps sending an arbitrary ring element to an invariant and an
invariant to itself. In fact, we obtain invariantization maps that are
also linear over a localized invariant ring. As hinted at above, using
extended Derksen ideals only yields a localization $R^G_a$ of an
invariant ring. We present a semi-algorithm for extracting the
original invariant ring $R^G$ from this, which terminates after
finitely many steps if and only if $R^G$ is finitely
generated. \sref{sFinite} is devoted to the case of finite groups. It
contains the algorithms mentioned above and also some examples,
emphasizing multiplicative invariants and invariants of linear
actions.

The remaining sections of the paper focus on the situation that a
linear algebraic group over an algebraically closed field acts on an
irreducible affine variety. Unfortunately, when using a tamely
extended Derksen ideal, the algorithm for computing a localization of
the invariant ring requires that the invariant field equals the field
of fractions of the invariant ring. This restriction is discussed in
\sref{sItalian}, and it is completely circumvented by an algorithm
given in the final section of the paper. Extended Derksen ideals have
algebraic, geometric, and computational aspects. The latter two
aspects are dealt with in Sections~\ref{sDerksenG}
and~\ref{sDerksenC}. \sref{sDerksenG} studies geometric
interpretations of extended and tamely extended Derksen
ideals. \sref{sDerksenC} gives algorithms for computing tamely
extended Derksen ideals in such a way that a maximal reduction in the
number of variables that need to be taken into the Gr\"obner basis
computation is achieved. It is by this reduction that extended Derksen
ideals boost the efficiency of computations. With this, the geometry,
computation, and applicability of tamely extended Derksen ideals are
quite well understood, while all these issues are still rather
mysterious for their nontame cousins. The various strands flow
together in an algorithm for computing a localization of an invariant
ring, whose result can be fed into the semi-algorithm mentioned
above. The algorithm takes a particularly simple form in the special
case of the additive group. In fact, the algorithm given by
\mycite{essen} appears as a special case of the algorithm from this
paper. The final section of the paper is devoted to optimizations that
apply to the case of linear group actions on a vector space. The
section also contains an algorithm for computing invariants of
reductive groups, which is an alternative to Derksen's algorithm (see
\mycite{Derksen:99}) and the algorithm by the
author~[\citenumber{kem.separating}]. Running times of the algorithms
are compared. \\

Throughout the article, a ring is understood to be commutative with an
identity element. An algebra $R$ over a ring $K$ is a ring $R$ that
contains $K$ as a subring, and $K[a_1 \upto a_n] \subseteq R$ denotes
the subalgebra generated by elements $a_i \in R$. A homomorphism of
$K$-algebras is understood to fix $K$. Moreover, $(a_1 \upto a_n)
\subseteq R$ stands for the ideal generated by the~$a_i$.

\section{Extended Derksen ideals: algebraic
  aspects} \label{sDerksenA}%

We start by introducing a generalized notion of a Derksen ideal. We
also introduce extended and tamely extended Derksen ideals.

\begin{defi} \label{dDerksen}%
  Let $G$ be a group acting on an algebra $S$ over a ring $K$ (which
  in many applications is a field) by automorphisms. Let $R = K[a_1
  \upto a_n] \subseteq S$ be a finitely generated, $G$-stable
  subalgebra and take indeterminates $y_1 \upto y_n$, on which $G$
  acts trivially.
  \begin{enumerate}
  \item \label{dDerksenA} The \df{Derksen ideal} (with respect to the
    $a_i$) is the intersection
    \[
    D_{a_1 \upto a_n}:= \bigcap_{\sigma \in G} \bigl(y_1 - \sigma
    \cdot a_1 \upto y_n - \sigma \cdot a_n\bigr) \subseteq S[y_1 \upto
    y_n].
    \]
  \item \label{dDerksenB} A proper, $G$-stable ideal $E \subsetneqq
    S[y_1 \upto y_n]$ is called an \df{extended Derksen ideal} (with
    respect to the $a_i$) if it contains $D_{a_1 \upto a_n}$.
  \item \label{dDerksenC} For an extended Derksen ideal $E$, consider
    the ideal
    \[
    I := \bigl\{f(a_1 \upto a_n) \mid f \in K[y_1 \upto y_n] \cap
    E\bigr\} \subseteq R.
    \]
    Then $E$ is called \df{tamely extended} if the intersection
    $\bigcap_{\sigma \in G} \sigma \cdot I$ is nilpotent.
  \end{enumerate}
\end{defi}

It is easy to see that $D_{a_1 \upto a_n}$ itself is a tamely extended
Derksen ideal (with $I = \{0\}$).
Before discussing geometric and computational aspects of our notions,
we prove the main results of this paper. The next five theorems will
all apply to the following situation:

\begin{assumption} \label{Assumption}%
  $G$ is a group acting on a field $L$ by automorphisms. We fix a
  subring $K \subseteq L^G$ and a finitely generated, $G$-stable
  $K$-subalgebra $R = K[a_1 \upto a_n] \subseteq L$ with $\Quot(R) =
  L$. Let $E \subseteq L[y_1 \upto y_n]$ be an extended Derksen ideal
  with respect to the $a_i$. Assume that $\mathcal{G}$ is a reduced
  Gr\"obner basis (see \mycite[Definition~5.29]{BW}) of $E$ with
  respect to an arbitrary monomial ordering, and let $A \subseteq L$
  be the $K$-subalgebra generated by the coefficients of all
  polynomials from $\mathcal{G}$.
\end{assumption}

Under this assumption, the ideals that are intersected when forming
$D_{a_1 \upto a_n}$ are maximal. It follows that if $G$ is finite,
then the only extended Derksen ideal is $D_{a_1 \upto a_n}$
itself.

Our first result is essentially (the first part of) Theorem~3.7 from
\mycite{HK:06}, with two differences: On the one hand, is is more
general since $L$ is not required to be a rational function field, but
on the other hand, it is more restricted since it does not extend to
rational actions. For instance, Example~4.2 from~[\citenumber{HK:06}]
cannot be dealt with by \tref{tInvariantField}.

\begin{theorem}[Invariant field] \label{tInvariantField}%
  Under the assumption~\ref{Assumption}, suppose that $E$ is a tamely
  extended Derksen ideal. Then
  \[
  L^G = \Quot(A).
  \]
\end{theorem}

We will prove the above theorem together with
Theorems~\ref{tLocalize} and~\ref{tInvariantization}.
Notice that \tref{tInvariantField}, just as the other results from
this section, requires no hypothesis on properties of the group action
(such as reductivity).

The following result deals with the computation of a localization of
the invariant ring $R^G$.

\begin{theorem}[Localized invariant ring] \label{tLocalize}%
  Under the assumption~\ref{Assumption} we have $R^G \subseteq A
  \subseteq L^G$.  If there exists a nonzero invariant $a \in R^G$
  with $A \subseteq R_a := R[a^{-1}]$, then
  \begin{equation} \label{eqLocalize}%
    R^G_a = A_a.
  \end{equation}
\end{theorem}

\begin{ex} \label{exC2}%
  Let $K$ be an arbitrary integral domain.
  \begin{enumerate}
    \renewcommand{\theenumi}{\arabic{enumi}}
  \item Consider the automorphism of the polynomial ring $K[x]$
    sending~$x$ to $1-x$. This generates a group $G \cong C_2$ whose
    Derksen ideal is
    \[
    D_x = \bigl(y - x\bigr) \cap \bigl(y - (1-x)\bigr) = \bigl(y^2 - y
    - (x^2 - x)\bigr) \subseteq L[y],
    \]
    where $L = \Quot(K[x])$. We already have a Gr\"obner basis, and
    the algebra generated by its coefficients is $A = K[x^2 - x]$,
    which happens to lie in $K[x]$. So we obtain
    \[
    K[x]^G = K[x^2 - x].
    \]
  \item \label{exC22} The automorphisms $x \to -x$ and $x \to x^{-1}$
    of the Laurent polynomial ring $K[x,x^{-1}]$ generate a group $G
    \cong C_2 \times C_2$. It is easy to see (using \pref{pFinite}
    below, for example) that the Derksen ideal is
    \[
    D_{x,x^{-1}} = \bigl(y_1^4 - (x^2 + x^{-2}) y_1^2 + 1,y_2 + y_1^3
    - (x^2 + x^{-2}) y_1\bigr) \subseteq L[y_1,y_2],
    \]
    where $L = \Quot(K[x])$. We obtain
    \[
    K[x,x^{-1}]^G = K\bigl[x^2 + x^{-2}\bigr].
    \]
  \end{enumerate}
  As these computations can easily be done by hand, it is not
  surprising that the results can also be verified directly quite
  easily. The significance of the example lies in the fact that it
  cannot be treated with the methods of \mycite{HK:06} or
  \mycite{Kamke:Kemper:2011} since ground ring $K$ need not be a field
  and in~\eqref{exC22} the action is not on a polynomial ring.
\end{ex}


If $G$ is finite, it is clear that $a \in R^G$ as in the
\tref{tLocalize} exists: Choose a common denominator of the generators
of $A$ and take~$a$ as its orbit product. In \sref{sItalian}, we will
discuss the existence of~$a$ in the case of algebraic groups. Once
again, in the above theorem $G$ is not assumed to be reductive, so the
theorem can (and does) provide finitely generated localizations of
nonfinitely generated invariant rings. This is an instance of the
following more general result, which can be found in
\mycite[Proposition~2.1(b)]{Giral:81} or
\mycite[Exercise~10.3]{Kemper.Comalg}: For every subalgebra $B
\subseteq A$ of a finitely generated domain $A$ over a ring there
exists a nonzero $a \in B$ such that $B_a$ is finitely generated.

We still assume the situation given by
Assumption~\ref{Assumption}. The Gr\"obner basis $\mathcal{G}$ induces
a normal form map $\NF_\mathcal{G}$. Using this, we define a new map
as follows: An element $b \in R$ can be written as $b = f(a_1 \upto
a_n)$ with $f \in K[y_1 \upto y_n]$. Define
\[
\mapl{\phi_\mathcal{G}}{R}{A}{b}{\bigl(\NF_\mathcal{G}(f)\bigr)(0
  \upto 0)}
\]
(i.e., set all $y_i$ equal to zero in the normal form of~$f$). We
call~$\phi_\mathcal{G}$ the \df{invariantization} map.

\begin{theorem}[Invariantization] \label{tInvariantization}%
  The invariantization map~$\phi_\mathcal{G}$ is a well-defined
  homomorphism of $R^G$-modules. It restricts to the identity on
  $R^G$. If~\eqref{eqLocalize} is satisfied, then~$\phi_\mathcal{G}$
  uniquely extends to an $R_a^G$-linear projection $R_a
  \twoheadrightarrow R_a^G$, and in particular $R_a^G$ is a direct
  summand of $R_a$.
\end{theorem}

The concept of invariantization was introduced by
\mycite{Fels:Olver:1999}, who used the term for a projection from the
set of smooth functions on an open subset of a manifold to the set of
local invariants under a group action. So the properties of our
map~$\phi_\mathcal{G}$ justify calling it invariantization. A
(different) algebraic version of invariantization was introduced by
\mycite{Hubert:Kogan:2007} in order to compute Fels and Olver's
invariantization in the case of algebraic functions (see Theorem~3.9
in~[\citenumber{Hubert:Kogan:2007}]).

\begin{proof}[Proof of
  Theorems~\ref{tInvariantField}, \ref{tLocalize}, and
  \ref{tInvariantization}]%
  $G$ acts on $L[y_1 \upto y_n]$ coefficient-wise. Hence for $\sigma
  \in G$ the set $\sigma \cdot \mathcal{G}$ is a reduced Gr\"obner
  basis of $\sigma \cdot E = E$. It follows from the uniqueness of
  reduced Gr\"obner bases (see \mycite[Theorem~5.43]{BW}) that $\sigma
  \cdot \mathcal{G} = \mathcal{G}$. Since the polynomials from
  $\mathcal{G}$ have pairwise distinct leading monomials, this implies
  that~$\sigma$ fixes every polynomial in $\mathcal{G}$, so
  $\mathcal{G} \subseteq L^G[y_1 \upto y_n]$. Hence $A \subseteq L^G$
  and $\Quot(A) \subseteq L^G$, which establishes two of the claimed
  inclusions in Theorems~\ref{tInvariantField} and~\ref{tLocalize}.

  Now let $b \in L^G$ and suppose that $E$ is tamely extended. The set
  $J := \{d \in R \mid b d \in R\} \subseteq R$ is a nonzero,
  $G$-stable ideal. Therefore $J \not\subseteq I$ (with $I$ the ideal
  from \dref{dDerksen}\eqref{dDerksenC}), since otherwise $J
  \subseteq \bigcap_{\sigma \in G} \sigma \cdot I$, which is nilpotent
  and therefore zero by hypothesis. So there exists $g \in K[y_1 \upto
  y_n] \setminus E$ such that $g(a_1 \upto a_n) \in J$. By the
  definition of $J$, this implies the existence of $f \in K[y_1 \upto
  y_n]$ with
  \begin{equation} \label{eqFrac}%
    b g(a_1 \upto a_n) = f(a_1 \upto a_n).
  \end{equation}
  Set $h := f - b g \in L[y_1 \upto y_n]$. For $\sigma \in G$, the
  $G$-invariance of~$b$ implies
  \[
  h(\sigma \cdot a_1 \upto \sigma \cdot a_n) = \sigma \cdot
  \bigl(f(a_1 \upto a_n) - b g(a_1 \upto a_n)\bigr) = 0.
  \]
  Therefore $h \in D_{a_1 \upto a_n} \subseteq E$ and, using the
  $L$-linearity of the normal form map, we conclude
  \begin{equation} \label{eqNormalForm}%
    0 = \NF_{\mathcal{G}}(h) = \NF_{\mathcal{G}}(f) - b
    \NF_{\mathcal{G}}(g).
  \end{equation}
  Since $g \notin E$ we have $\NF_{\mathcal{G}}(g) \ne 0$,
  so~\eqref{eqNormalForm} implies
  \[
  b = \frac{\NF_{\mathcal{G}}(f)}{\NF_{\mathcal{G}}(g)}.
  \]
  Since~$f$, $g$, and $\mathcal{G}$ are contained in $A[y_1 \upto
  y_n]$, we can see from the algorithm for computing normal forms (see
  \mycite[Algorithm~9.8]{Kemper.Comalg}) that also
  $\NF_{\mathcal{G}}(f),\NF_{\mathcal{G}}(g) \in A[y_1 \upto y_n]$, so
  the above equation tells us
  \[
  b \in \Quot(A[y_1 \upto y_n]) \cap L = \Quot(A).
  \]
  This completes the proof of \tref{tInvariantField}.

  For the proof of \tref{tLocalize} let $b \in R^G$,
  so~\eqref{eqFrac} holds with $g = 1$. Since $E$ is a proper ideal,
  we have have $\NF_{\mathcal{G}}(1) = 1$ and~\eqref{eqNormalForm}
  yields
  \begin{equation} \label{eqNFF}
    b = \NF_{\mathcal{G}}(f) \in A[y_1 \upto y_n] \cap L = A.
  \end{equation}
  Now both inclusions $R^G \subseteq A \subseteq L^G$ are
  established. Assume $A \subseteq R_a$ with $a \in R^G$ nonzero. Then
  \[
  R^G_a \subseteq A_a \subseteq R_a \cap L^G = R_a^G,
  \]
  so \tref{tLocalize} is established.

  We now turn our attention to \tref{tInvariantization}. Let $f \in
  K[y_1 \upto y_n]$ be a polynomial and $b := f(a_1 \upto a_n) \in
  R$. If $b = 0$, then $\NF_\mathcal{G}(f) = 0$
  by~\eqref{eqNFF}. This implies that $\phi_\mathcal{G}$ does not
  depend on the choice of the polynomial in $K[y_1 \upto y_n]$ that is
  used for its definition. If $b \in R^G$, it follows
  from~\eqref{eqNFF} that $\phi_\mathcal{G}(b) = b$. For $b \in R$
  not necessarily an invariant, we have already seen that
  $\NF_\mathcal{G}(f) \in A[y_1 \upto y_n]$, so $\phi_\mathcal{G}(b)
  \in A$.

  To prove that~$\phi_\mathcal{G}$ is a homomorphism of $R^G$-modules,
  take a further element $c = g(a_1 \upto a_n) \linebreak \in R \setminus \{0\}$
  such that $\frac{b}{c} \in L^G$. By~\eqref{eqNormalForm}, this
  implies $\NF_\mathcal{G}(f) = \frac{b}{c}\NF_\mathcal{G}(g)$, and
  therefore $\phi_\mathcal{G}(b) = \frac{b}{c}\phi_\mathcal{G}(c)$. In
  particular, $\phi_\mathcal{G}(r c) = r \phi_\mathcal{G}(c)$ for $r
  \in R^G$, and this also holds if $c = 0$. The additivity
  of~$\phi_\mathcal{G}$ follows from the additivity of the normal
  form. Finally, if~\eqref{eqLocalize} holds, then it is clear that
  $a^{-k} b \mapsto a^{-k} \phi_\mathcal{G}(b)$ gives a well-defined
  map that uniquely extends~$\phi_\mathcal{G}$ to an $R_a^G$-linear
  map. It also follows that this extension is the identity on $R_a^G$
  and that its image is $R_a^G$.
\end{proof}

We now give two examples of (extended) Derksen ideals and the
corresponding invariantization maps.

\begin{ex} \label{exInvariantization}%
  In this example $K$ is assumed to be a field, and algebraically
  closed in~\eqref{exInvariantization1}.
  \begin{enumerate}
    \renewcommand{\theenumi}{\arabic{enumi}}
  \item \label{exInvariantization1} Consider the action of the
    multiplicative group $G = \Gm$ on the polynomial ring $R =
    K[x_1,x_2]$ with weight $(1,-1)$. With $L = K(x_1,x_2)$, the
    Derksen ideal is
    \[
    D_{x_1,x_2} = \bigl(y_1 y_2 - x_1 x_2\bigr) \subseteq L[y_1,y_2],
    \]
    with a reduced Gr\"obner basis $\mathcal{G}$ already
    displayed. Since $y_1$ and $y_2$ are their own normal forms, we
    get $\phi_\mathcal{G}(x_1) = \phi_\mathcal{G}(x_2) = 0$. However,
    $y_1 y_2$ has the normal form $x_1 x_2$, so $\phi_\mathcal{G}(x_1
    x_2) = x_1 x_2$. (This also follows since $x_1 x_2$ is an
    invariant.) We see that $\phi_\mathcal{G}$ is not in general a
    homomorphism of rings. We can also consider the extended Derksen
    ideal
    \[
    E = D_{x_1,x_2} + \bigl(y_1 - 1\bigr) = \bigl(y_1 - 1,y_2 - x_1
    x_2\bigr) \subseteq L[y_1,y_2].
    \]
    Using the reduced Gr\"obner basis $\mathcal{G}'$ of this, we
    obtain $\phi_{\mathcal{G}'}(x_1) = 1$. This shows that the
    invariantization map may depend on the choice of the extended
    Derksen ideal.
  \item Consider the finite symmetric group $G = S_2$ with its natural
    action on $R = K[x_1,x_2]$. With $L = K(x_1,x_2)$, the Derksen
    ideal is
    \[
    D_{x_1,x_2} = \bigl(y_1 + y_2 - x_1 - x_2,y_2^2 - (x_1 + x_2) y_2
    + x_1 x_2\bigr) \subseteq L[y_1,y_2].
    \]
    With respect to a monomial ordering with $y_1 > y_2$, the
    displayed basis $\mathcal{G}$ is the reduced Gr\"obner basis. We
    obtain $\phi_\mathcal{G}(x_1) = x_1 + x_2$ and
    $\phi_\mathcal{G}(x_2) = 0$. This shows that~$\phi_\mathcal{G}$ is
    not $G$-equivariant, and in particular, it does not coincide with
    the Reynolds operator (which exists if $\ch(K) \ne 2$). Moreover,
    if we had chosen a monomial ordering with $y_2 > y_1$, the
    Gr\"obner basis would have changed in such a way that $x_1$ would
    be sent to~$0$ and $x_2$ to $x_1 + x_2$. So we see that even when
    one fixes an extended Derksen ideal, the invariantization map may
    depend on the chosen monomial ordering. \exend
  \end{enumerate}
  \renewcommand{\exend}{}
\end{ex}

See \rref{rInvariantization} for a variant of the invariantization
map that does not depend on the choice of the monomial ordering.

Suppose that in the situation of \tref{tLocalize} the
equality~\eqref{eqLocalize} holds. Then by \tref{tInvariantization},
$R_a^G$ is a direct summand of $R_a$. Now we can use a result by
\mycite{Hochster:Huneke:1995}, which tells us that if $R_a$ is a
regular ring and if $R_a^G$ contains a field, then $R_a^G$ is
Cohen--Macaulay. So \tref{tInvariantization} has the following
consequence:

\begin{cor} \label{cCM}%
  Under the assumption~\ref{Assumption} suppose that there exists a
  nonzero $a \in R^G$ with $A \subseteq R_a$. If $R_a$ is regular and
  $R_a^G$ contains a field, then $R_a^G$ is Cohen--Macaulay.

  Notice that all the hypotheses of this corollary are satisfied if
  $G$ is a linear algebraic group which has no surjective homomorphism
  onto $\Gm$ (see \tref{tItalian}\eqref{tItalianB}) and $R = K[V]$
  with $V$ a $G$-module.
\end{cor}

In the situation given by Assumption~\ref{Assumption}, the first
inclusion from \tref{tLocalize} tells us that every invariant from
$R^G$ can be written as a polynomial (over $K$) in the coefficients
occurring in the polynomials from $\mathcal{G}$. The next result deals
with finding such a polynomial explicitly. Instead of considering the
set $\mathcal{S}$ of all coefficients of polynomials in $\mathcal{G}$,
it is often useful to choose some invariants~$b_i$ such that all
elements of $\mathcal{S}$ can be expressed as polynomials in
the~$b_i$. More formally, assume that we have a map $\map{\psi}{K[t_1
  \upto t_r]}{L^G}$ of $K$-algebras, with $K[t_1 \upto t_r]$ a
polynomial ring, whose image contains $A$. Choosing preimages of all
coefficients of the polynomials in $\mathcal{G}$, we form a set
$\mathcal{G}_t \subseteq K[t_1 \upto t_r,y_1 \upto y_n]$ such that
$\psi(\mathcal{G}_t) = \mathcal{G}$. (Here we extend~$\psi$ to $K[t_1
\upto t_r,y_1 \upto y_n]$ by sending each $y_i$ to itself.) The
following theorem gives an invariance test and an algorithm for
rewriting an invariant in terms of the $b_i := \psi(t_i)$.

\begin{theorem}[Rewriting invariants] \label{tRewrite}%
  With the above notation, let $b = f(a_1 \upto a_n) \in R$ with $f
  \in K[y_1 \upto y_n]$, let $g \in K[t_1 \upto t_r,y_1 \upto y_n]$ be
  a normal form of~$f$ with respect to $\mathcal{G}_t$, and obtain
  $g_0 \in K[t_1 \upto t_r]$ by setting all $y_i$ equal to zero
  in~$g$. Then $b \in R^G$ if and only if $b = \psi(g_0)$. In this
  case, $\psi(g_0)$ expresses~$b$ as a polynomial in the invariants
  $b_i = \psi(t_i)$.
\end{theorem}

\begin{proof}
  It is easy to see that $\psi(g) = \NF_\mathcal{G}(f)$. This implies
  $\psi(g_0) = \phi_\mathcal{G}(f)$, so \tref{tRewrite} follows from
  \tref{tInvariantization}.
\end{proof}

For invariant fields, similar rewriting algorithms were given by
\mycite{HK:06} and \mycite{Kemper.Rosenlicht.07}.

We now come back to \tref{tLocalize} and ask how it can be used to
calculate $R^G$. If~\eqref{eqLocalize} holds, we can assume~$a$ to be
one of the generators of $A$ and then multiply every generator of $A$
by a suitable power of~$a$ to obtain an element of $R^G$. This
produces a subalgebra $B \subseteq R^G$ that still satisfies $R_a^G =
B_a$. The following theorem deals with how to extract the invariant
ring $R^G$ in this situation.

\begin{theorem}[Invariant ring] \label{tInvariantRing}%
  In the situation~\ref{Assumption}, assume that $B \subseteq R^G$ is
  a finitely generated $K$-subalgebra such that $R_a^G = B_a$ with $a
  \in B$ nonzero. Define an ascending chain of subalgebras
  $B_0,B_1,B_2, \ldots \subseteq R$ by setting $B_0 := B$ and taking
  $B_{k+1}$ to be the subalgebra generated by $a^{-1} B_k \cap
  R$. Then
  \[
  R^G = \bigcup_{k=0}^\infty B_k.
  \]
  If $K$ is Noetherian, then all $B_k$ are finitely generated (as
  $K$-algebras). If $B_k = B_{k+1}$ for some~$k$, then $R^G = B_k$. If
  $R^G$ is finitely generated then such a~$k$ exists.
\end{theorem}

\begin{proof}
  We consider the sets $B^{(k)} := a^{-k} B \cap R$ and show by
  induction on~$k$ that $B^{(k)} \subseteq B_k \subseteq R^G$ and that
  $B_k$ is finitely generated if $K$ is Noetherian. This is true for
  $k = 0$, and, using the induction hypothesis, we have
  \[
  B^{(k+1)} = a^{-1}\bigl(a^{-k} B \cap R \cap a R\bigr) = a^{-1}
  B^{(k)} \cap R \subseteq a^{-1} B_k \cap R \subseteq B_{k+1}.
  \]
  Moreover, $B_k \subseteq R^G$ implies $a^{-1} B_k \cap R \subseteq
  L^G \cap R = R^G$, so $B_{k+1} \subseteq R^G$. Finally, if $K$ is
  Noetherian, then so is $B_k$. Therefore $B_k \cap a R \subseteq B_k$
  is a finitely generated ideal, and so $a^{-1} B_k \cap R$ is a
  finitely generated $B_k$-module. This implies that $B_{k+1}$ is
  finitely generated as a $B_k$-algebra and therefore also as a
  $K$-algebra.

  Now we can see that our hypothesis $R^G_a = B_a$ implies
  \[
  R^G \subseteq \bigcup_{k=0}^\infty B^{(k)} \subseteq
  \bigcup_{k=0}^\infty B_k \subseteq R^G.
  \]
  Clearly $B_k = B_{k+1}$ implies $B_k = B_i$ for all $i \ge k$ and
  therefore $B_k = R^G$. Now suppose that $R^G$ is finitely
  generated. Then all generators of $R^G$ are contained in some $B_k$,
  so
  \[
  R^G \subseteq B_k \subseteq B_{k+1} \subseteq R^G.
  \]
  This finishes the proof.
\end{proof}

Now we turn \tref{tInvariantRing} into a procedure. Unsurprisingly,
this will involve Gr\"obner basis computations. These are possible if
$K$ is a field, but also over certain rings. In fact, we need to
assume that $K$ is Noetherian and that there is an algorithm for
computing all solutions $(c_1 \upto c_r) \in K^r$ of a linear equation
$a_1 c_1 + \cdots + a_r c_r = b$ with $b,a_i \in K$ (see
\mycite[Algorithm~4.2.1]{AL}; Chapter~4 of this book gives a nice
introduction to Gr\"obner bases over rings). Rings with these
properties are sometimes called \df{Zacharias rings} (see, for example
\mycite{Mora:2005}). As an example, all Euclidean rings are Zacharias
rings. As the usual Buchberger algorithm, Algorithm~4.2.1
from~[\citenumber{AL}] can easily be modified to give representations
of the Gr\"obner basis elements as linear combinations of the original
basis elements. By means of Gr\"obner bases one also gets a test for
membership in an ideal $I \subseteq K[y_1 \upto y_n]$ in a polynomial
ring (see~[\citenumber{AL}, Theorem~4.1.12 and Corollary~4.1.16]) and
an algorithm for computing elimination ideals $K[y_1 \upto y_k] \cap
I$ (see~[\citenumber{AL}, Theorem~4.3.6]). With this, our toolbox of
Gr\"obner basis techniques is prepared for turning
\tref{tInvariantRing} into a procedure.

Since the sequence of subalgebras $B_k$ in the theorem terminates only
if the invariant ring is finitely generated, the procedure is a
semi-algorithm in the sense that it need not terminate after finitely
many steps. Semi-algorithm~\ref{aSemi} has appeared in a less explicit
and less general form in \mycite{essen}.

\begin{salg}[``Unlocalizing'' the invariant ring] \label{aSemi} \mbox{}%
  \begin{description}
  \item[\bf Input:] Given the situation of
    Assumption~\ref{Assumption}, the procedure needs:
    \begin{itemize}
    \item the kernel $I$ of the map $K[y_1 \upto y_n] \to R$, $y_i
      \mapsto a_i$,
    \item a subalgebra $B \subseteq R$ generated by elements
      $\overline{f}_i := f_i(a_1 \upto a_n)$ with $f_i \in K[y_1
      \upto y_n]$, $i = 1 \upto k$, and
    \item a nonzero element $a \in B$ such that $R^G_a = B_a$.
    \end{itemize}
    It is required that $K$ is a Zacharias ring.
  \item[\bf Output:] Generators of $R^G$ as a $K$-algebra. The ideal
    of relations between the generators is also computed. The
    procedure terminates after finitely many steps if and only if
    $R^G$ is finitely generated.
  \end{description}
  \begin{enumerate}
    \renewcommand{\theenumi}{\arabic{enumi}}
  \item Set $m := k$.
  \item \label{aSemi2} This step is optional. Substitute
    $\{\overline{f}_1 \upto \overline{f}_m\}$ by a (smaller) subset of
    $R$ generating the same subalgebra.
  \item \label{aSemi3} With additional indeterminates $z_1 \upto z_m$,
    let $\widehat{L} \subseteq K[y_1 \upto y_n,z_1 \upto z_m]$ be the
    ideal generated by $I$ and $z_i - f_i$ ($i = 1 \upto
    m$). Moreover, let $\widehat{M} \subseteq K[y_1 \upto y_n,z_1
    \upto z_m]$ be the ideal generated by $\widehat{L}$ and a
    polynomial $g \in K[z_1 \upto z_m]$ with $g(\overline{f}_1 \upto
    \overline{f}_m) = a$.
  \item Compute the elimination ideal $M := K[z_1 \upto z_m] \cap
    \widehat{M}$. Let $L \subseteq K[z_1 \upto z_m]$ be the ideal
    generated by the elimination ideal $J := K[z_1 \upto z_m] \cap
    \widehat{L}$ and by~$g$.
  \item \label{aSemi5} If $M \subseteq L$, return $\overline{f}_1
    \upto \overline{f}_m$ as the desired generators for $R^G$ and
    return $J$ as the ideal of relations between them. The inclusion
    $M \subseteq L$ can be tested by performing membership tests on
    the generators of $M$.
  \item \label{aSemi6} Choose $h_1 \upto h_r \in M$ such that $L$
    together with the $h_i$ generates $M$. Compute a Gr\"obner basis
    of the ideal $J' \subseteq K[y_1 \upto y_n]$ generated by $I$ and
    $g(f_1 \upto f_m)$, together with representations of the Gr\"obner
    basis elements as linear combinations of the original basis
    elements.
  \item \label{aSemi7} For each $i \in \{1 \upto r\}$, reduce $h_i(f_1
    \upto f_m)$ with respect to the Gr\"obner basis of $J'$ (which
    will yield zero) and then express $h_i(f_1 \upto f_m)$ as a linear
    combination of the ideal basis of $I$ and $g(f_1 \upto f_m)$. This
    yields $f_{m+i} \in K[y_1 \upto y_n]$ such that
    \begin{equation} \label{eqFmi}%
      g(\overline{f}_1 \upto \overline{f}_m) \cdot \overline{f}_{m+i}
      = h_i(\overline{f}_1 \upto \overline{f}_m).
    \end{equation}
  \item \label{aSemi8} Set $m := m + r$ and go to step~\ref{aSemi2}.
  \end{enumerate}
\end{salg}

\begin{proof}[Proof of correctness of Semi-algorithm~\ref{aSemi}]
  Let $\widehat{B} = K\bigl[\overline{f}_1 \upto \overline{f}_m\bigr]
  \subseteq R$, with $f_1 \upto f_m$ as in step~\ref{aSemi3}, possibly
  after having performed steps~\ref{aSemi2}--\ref{aSemi8} several
  times. For $h \in K[z_1 \upto z_m]$, it is easy to verify the
  equivalences
  \begin{equation} \label{eqEquiv}%
    h(\overline{f}_1 \upto \overline{f}_m) \in aR \quad
    \Longleftrightarrow \quad h \in M
  \end{equation}
  and
  \[
  h(\overline{f}_1 \upto \overline{f}_m) \in a \widehat{B} \quad
  \Longleftrightarrow \quad h \in L.
  \]
  It follows that $M/L \cong \bigl(\widehat{B} \cap a R\bigr)/a
  \widehat{B} \cong \bigl(a^{-1} \widehat{B} \cap
  R\bigr)/\widehat{B}$, so step~\ref{aSemi5} tests $a^{-1} \widehat{B}
  \cap R \subseteq \widehat{B}$. If this is not the case, the
  algorithm reaches step~\ref{aSemi6}, and~\eqref{eqEquiv} guarantees
  the existence of polynomials $f_{m+i} \in K[y_1 \upto y_n]$
  satisfying~\eqref{eqFmi}. This implies $h_i(f_1 \upto f_m) \in J'$,
  so the reduction of $h_i(f_1 \upto f_m)$ will yield zero, as
  claimed, and $f_{m+i}$ is found in
  step~\ref{aSemi7}. From~\eqref{eqFmi} and the choice of $h_i$ in
  step~\ref{aSemi6} it follows that the $f_{m+i} + I$ generate $a^{-1}
  \widehat{B} \cap R$ as a $\widehat{B}$-module. So the algorithm goes
  back to step~\ref{aSemi2} with $\widehat{B}$ replaced by the algebra
  generated by $a^{-1} \widehat{B} \cap R$. It follows that the
  algorithm produces the same ascending chain of subalgebras $B_k$ of
  $R$ that is dealt with in \tref{tInvariantRing}, and the correctness
  of the termination condition in step~\ref{aSemi5} follows from that
  theorem.
\end{proof}

\begin{rem*}
  In fact, Semi-algorithm~\ref{aSemi} (and \tref{tInvariantRing}) work
  in the following, more general situation: $R$ is a finitely
  generated algebra over a Zacharias ring $K$, and $a \in B \subseteq
  R$ is an element of a finitely generated subalgebra such that
  multiplication by~$a$ is injective on $R$. Then the procedure
  computes $R \cap B_a$. 
\end{rem*}

\section{Finite groups} \label{sFinite}

In this section we consider the special case of finite groups. In the
situation given by Assumption~\ref{Assumption}, let us assume that $G$
is finite. This has the following beneficial consequences:

\begin{itemize}
\item Since the Derksen ideal is a finite intersection of ideals in
  $L[x_1 \upto x_n]$, we have an algorithm for computing it (see
  \mycite[Algorithm~6.3]{BW}).
\item From the Derksen ideal we can compute a localization $R^G_a$ of
  the invariant ring by using \tref{tLocalize}, since a nonzero
  invariant $a \in R^G$ as in the theorem exists (see the remark after
  \exref{exC2}).
\item If $K$ is Noetherian, then $R^G$ is finitely generated by
  \mycite{noe:c}. So Semi-algorithm~\ref{aSemi} will terminate after
  finitely many steps and compute the invariant ring $R^G$, provided
  that $K$ is a Zacharias ring.
\end{itemize}

So we get the following algorithm:

\begin{alg}[Invariant ring of a finite group acting on a domain over a
  Zacharias ring] \label{aFinite} \mbox{}%
   \begin{description}
   \item[\bf Input:] A prime ideal $I \subset K[x_1 \upto x_n]$ in a
     polynomial ring over a Zacharias ring $K$ (e.g., a Euclidean
     ring) with $K \cap I = \{0\}$, and a finite group of
     automorphisms of $R := K[x_1 \upto x_n]/I$.
  \item[\bf Output:] A finite set of generators of the invariant ring
    $R^G$ as a $K$-algebra.
  \end{description}
  \begin{enumerate}
    \renewcommand{\theenumi}{\arabic{enumi}}
  \item \label{aFinite1} With $a_i := x_i + I \in R$ and $L :=
    \Quot(R)$, compute a reduced Gr\"obner basis $\mathcal{G}$ of the
    Derksen ideal $D_{a_1 \upto a_n} \subseteq L[y_1 \upto y_n]$. This
    can be done by using Algorithm~6.3 from~[\citenumber{BW}].
  \item \label{aFinite2} Choose a nonzero invariant $a \in R^G$ such
    that $a^k \cdot \mathcal{G} \subseteq R[y_1 \upto y_n]$ holds for
    some~$k$.
  \item Let $S$ be the set whose elements are the coefficients of all
    polynomials in $\mathcal{G}$. Change $S$ by multiplying each
    element by a suitable power of~$a$ to obtain an element from
    $R$. Let $B \subseteq R^G$ be the subalgebra generated by~$a$ and
    $S$.
  \item Apply Semi-algorithm~\ref{aSemi} to obtain the desired
    generators of $R^G$. This is guaranteed to terminate after
    finitely many steps.
  \end{enumerate}
\end{alg}

Since it is necessary to compute a Gr\"obner basis of the Derksen
ideal, it may seem that already step~\ref{aFinite1} of the algorithm
cannot be ``controlled.'' However, the following proposition shows
that under the very mild hypothesis that $K$ is infinite, the
Gr\"obner basis is in fact very much under control. Indeed, if $K$ is
infinite (which is always the case if $K$ is not a field), then by
picking a suitable $K$-linear combination of the generators $a_i$ we
can produce an element of $R$ that is fixed by no other element from
$G$ but the identity. By taking this as an additional generator (if
necessary) we can therefore achieve that one of the generators, say
$a_1$, has trivial stabilizer subgroup $G_{a_1} \subseteq G$. So the
following proposition is applicable if $K$ is infinite.

\begin{prop} \label{pFinite}%
  In the situation given by Assumption~\ref{Assumption}, assume that
  $G$ is finite and $a_1$ has trivial stabilizer $G_{a_1} =
  \{\id\}$. Then the polynomials
  \[
  f_1 = \prod_{\sigma \in G} (y_1 - \sigma \cdot a_1) \quad \text{and}
  \quad f_i = y_i - \sum_{\sigma \in G} (\sigma \cdot a_i) \cdot
  \!\!\!  \prod_{\tau \in G \setminus \{\sigma\}} \frac{y_1 - \tau
    \cdot a_1}{\sigma \cdot a_1 - \tau \cdot a_1} \quad (2 \le i \le
  n)
  \]
  form a reduced Gr\"obner basis of $D_{a_1 \upto a_n} \subseteq L[y_1
  \upto y_n]$ with respect to every monomial ordering with $x_1^k <
  x_i$ for $i,k > 1$. With $a := \discr(f_1) \in R^G \setminus \{0\}$
  (the discriminant of $f_1$ as a polynomial in $y_1$), we obtain $a
  \cdot f_i \in R[y_1 \upto y_n]$ for all~$i$.
\end{prop}

\begin{proof}
  Direct computation shows that $f_i(\rho \cdot a_1 \upto \rho \cdot
  a_n) = 0$ for every $\rho \in G$ and $1 \le i \le n$, so $f_i \in
  D_{a_1 \upto a_n}$. The linear map $L[y_1 \upto y_n] \to L^{|G|}$
  given by $f \mapsto \bigl(f(\sigma \cdot a_1 \upto \sigma \cdot a_n)
  | \sigma \in G\bigr)$ is surjective and has the kernel $D_{a_1 \upto
    a_n}$, so $L[y_1 \upto y_n]/D_{a_1 \upto a_n}$ shares the
  $L$-dimension $|G|$ with $L[y_1 \upto y_n]/(f_1 \upto f_n)$. It
  follows that $D_{a_1 \upto a_n} = (f_1 \upto f_n)$. By Buchberger's
  first criterion (see \mycite[Theorem~5.68]{BW}), the $f_i$ form a
  Gr\"obner basis, which is clearly reduced.
\end{proof}

\pref{pFinite} leads to a variant of steps~\ref{aFinite1}
and~\ref{aFinite2} in \aref{aFinite}. In this variant it is often
possible to choose~$a$ as a proper divisor of $\discr(f_1)$. This will
be very beneficial for computations. Since the discriminant is the
product of all $\sigma \cdot a_1 - \tau \cdot a_1$, it is not hard to
choose a good (or even optimal) $a$ among its divisors. The
invariant~$a$ is also significant since $R_a^G$ is a direct summand of
$R_a$, which tends to imply that $R_a^G$ inherits nice geometric
properties from $R_a$. (An instance of this tendency is \cref{cCM}.)
In other words, the spectrum of $R^G$ should be nice outside of the
hypersurface given by~$a$. Now if~$a$ is chosen as in \pref{pFinite},
then a point outside the hypersurface given by~$a$ is fixed by no
other group element than the identity. This means that a point whose
stabilizer subgroup is trivial should map to a nice point in the
categorical quotient. An instance of this general philosophy is the
fact that if a finite group $G$ acts linearly on a finite-dimensional
vector space $V$, then a point $x \in V$ with $G_x = \{\id\}$ maps to
a regular point in $V \quo G$ (see, for example, \mycite{kem:loci}).

A particularly interesting special case to which \aref{aFinite} can be
applied is the case of multiplicative invariants: One considers a
subgroup $G \subseteq \GL_n(\ZZ)$ acting on the Laurent polynomial
ring $K[x_1 \upto x_n,x_1^{-1} \upto x_n^{-1}]$ by transforming the
exponent vectors of monomials in the obvious way. A very useful source
on multiplicative invariant theory is the book by
\mycite{Lorenz:2005}. Proposition~3.3.1 in that book reduces the
general situation to the case that $G$ is finite and $K = \ZZ$. So our
algorithm applies. Another interesting case is the action of a finite
subgroup $G \subseteq \GL_n(K)$ on a polynomial ring $K[x_1 \upto
x_n]$ by linear transformations of the $x_i$. In both these special
cases (multiplicative and linear actions), most of the optimizations
of Semi-algorithm~\ref{aSemi} that will be discussed in \sref{sLinear}
are applicable, since we have methods for producing ever more
invariants. Probably the most important ground ring is $K = \ZZ$,
since it allows the computation of multiplicative invariants but also
since ground rings that are finitely generated $\ZZ$-algebras can be
dealt with by regarding $\ZZ$ as the ground ring.

For $K = \ZZ$, \aref{aFinite} and its optimizations for linear and
multiplicative actions were implemented in MAGMA (see \mycite{magma})
by the author. So it is time now to present some examples.

\begin{ex} \label{exZ2}%
  Perhaps the simplest example of an invariant ring that is not
  Cohen--Macaulay (and that also violates Noether's degree bound) is
  the invariant ring $K[x_1,x_2,x_3,y_1,y_2,y_3]^G$ with $G \cong C_2$
  acting by interchanging the $x_i$ and $y_i$ and $K$ a field of
  characteristic~$2$ (see \mycite[Example~3.4.3]{Derksen:Kemper}). So
  it will be interesting to consider the invariant ring over $K =
  \ZZ$. Running the MAGMA program implementing \aref{aFinite} yields
  \[
  R^G := \ZZ[x_1,x_2,x_3,y_1,y_2,y_3]^G =
  \ZZ[s_1,s_2,s_3,p_1,p_2,p_3,u_{1,2},u_{1,3},u_{2,3},f]
  \]
  with
  \[
  s_i = x_i + y_i, \ p_i = x_i y_i, \ u_{i,j} = x_i y_j + y_i x_j, \
  \text{and} \ f = x_1 y_2 y_3 + y_1 x_2 x_3.
  \]
  If $K$ is a field (or, in fact, any other commutative ring), then
  $K[x_1,x_2,x_3,y_1,y_2,y_3]^G = K \otimes_\ZZ R^G$, so the above
  generators also generate the invariant ring over $K$.
\end{ex}

The next example deals with multiplicative invariants.

\begin{ex} \label{exMultiplicative}%
  The finite subgroups of $\GL_2(\ZZ)$ have been classified and can be
  found (up to conjugacy) in Table~1.2 from
  \mycite{Lorenz:2005}. Using the MAGMA program mentioned above, we
  computed the multiplicative invariant rings of all these groups. The
  results are in perfect agreement with Table~3.1
  from~[\citenumber{Lorenz:2005}].  The computations took just a few
  seconds or less, except for one case that took almost two hours. We
  present three examples. The action is always on the Laurent
  polynomial ring $\ZZ[x,y,x^{-1},y^{-1}]$.
  \begin{enumerate}
    \renewcommand{\theenumi}{\arabic{enumi}}
  \item 
    The matrix $\sigma := \left(\begin{smallmatrix} -1 & 0 \\ 0 &
        -1 \end{smallmatrix}\right)$ affords the transformation $x \to
    x^{-1}$, $y \to y^{-1}$. The invariant ring is
    \[
    \ZZ[x,y,x^{-1},y^{-1}]^{\langle \sigma \rangle} = \ZZ\bigl[x +
    x^{-1},y + y^{-1},x y^{-1} + x^{-1} y\bigr].
    \]
    The computation was done by the MAGMA program mentioned above and
    took 0.04 seconds on a 2.67 GHz Intel Xeon X5650 processor.
  \item 
    The matrix $\tau := \left(\begin{smallmatrix} 0 & 1 \\ 1 &
        0 \end{smallmatrix}\right)$ together with the above~$\sigma$
    generates a group that is isomorphic to $C_2 \times C_2$. So we
    are considering the additional transformation $x \leftrightarrow
    y$. The invariant ring is
    \[
    \ZZ[x,y,x^{-1},y^{-1}]^{\langle \sigma,\tau \rangle} = \ZZ\bigl[x
    y^{-1} + x^{-1} y,x y + x^{-1} y^{-1},x + y + x^{-1} +
    y^{-1}\bigr].
    \]
    The computation took 0.12 seconds.
  \item 
    The matrix $\rho := \left(\begin{smallmatrix} -1 & 0 \\ 0 &
        1 \end{smallmatrix}\right)$ together with the above~$\tau$
    generates a group that is isomorphic to the dihedral group of
    order~$8$. The invariant ring is
    \[
    \ZZ[x,y,x^{-1},y^{-1}]^{\langle \sigma,\tau \rangle} = \ZZ\bigl[x
    + y + x^{-1} + y^{-1},x y + x y^{-1} + x^{-1} y + x^{-1} y^{-1}
    \bigr].
    \]
    The computation took 0.04 seconds.
  \end{enumerate}
  In the third example the invariant ring is isomorphic to a
  polynomial ring, and in the others the generators are subject to
  just one relation.
\end{ex}

Noether's finiteness theorem~[\citenumber{noe:c}] asserts the finite
generation of invariant rings $R^G$ of finite groups without requiring
$R$ to be an integral domain. So it is somewhat unsatisfactory that
\aref{aFinite} has this requirement. We will now give an algorithm
that makes Noether's original argument constructive in the case of a
Zacharias ground ring. The algorithm is much less efficient and also
harder to implement than \aref{aFinite} which can, in fact, be used to
enhance the efficiency of the algorithm. We first present the
algorithm and then make comments on how the steps can be put into
practice.

\begin{alg}[Invariant ring of a finite group acting on an algebra over
  a Zacharias ring] \label{aNondomain} \mbox{}%
   \begin{description}
   \item[\bf Input:] A finitely generated algebra $R = K[a_1 \upto
     a_n]$ over a Zacharias ring $K$, and a finite group $G$ of
     automorphisms of $R$. $R$ need not be an integral domain.
   \item[\bf Output:] A finite set of generators of the invariant ring
     $R^G$ as a $K$-algebra.
  \end{description}
  \begin{enumerate}
    \renewcommand{\theenumi}{\arabic{enumi}}
  \item \label{aNondomain1} Form a subalgebra $A = K[b_1 \upto b_m]
    \subseteq R^G$ such that
    \begin{equation} \label{eqNoether}%
      \prod_{\sigma \in G} (y - \sigma \cdot a_i) \in A[y] \quad
      \text{for} \quad i = 1 \upto n.
    \end{equation}
    With $P := K[y_1 \upto y_m]$ a polynomial ring, the homomorphism
    $P \to R$, $y_i \mapsto b_i$ makes $R$ into a $P$-module.
  \item \label{aNondomain2} Choose $c_1 \upto c_r \in R$ that generate
    $R$ as an $A$-module, with $c_1 = 1$. This yields a surjective,
    $P$-linear map $\mapl{\phi}{P^r}{R}{(g_1 \upto g_r)}{\sum_{i=1}^r
      g_i(b_1 \upto b_m) c_i}$.
  \item \label{aNondomain3} Compute generators of $\ker(\phi)$. This
    defines a $P$-linear map $\map{\eta}{P^m}{P^r}$ such that the
    sequence $P^m \stackrel{\eta}{\longrightarrow} P^r
    \stackrel{\phi}{\longrightarrow} R \to 0$ is exact.
  \item \label{aNondomain4} With $G$ generated by elements $\sigma_1
    \upto \sigma_s$, define the map
    \[
    \mapl{\underline{\sigma} - \underline{\id}}{R}{R^s}{a}{(\sigma_1
      \cdot a - a \upto \sigma_s \cdot a - a)},
    \]
    which is $P$-linear and has kernel $R^G$. Construct a $P$-linear
    map $\map{\psi}{P^r}{P^{r s}}$ such that
    \[
    (\underline{\sigma} - \underline{\id}) \circ \phi = \phi^{\oplus
      s} \circ \psi,
    \]
    where $\map{\phi^{\oplus s}}{P^{r s}}{R^s}$ is the component-wise
    application of~$\phi$. So the corresponding part of the diagram
    below commutes.
  \item \label{aNondomain5} Compute a generating set $m_1 \upto m_t$
    of the $P$-module
    \[
    M := \bigl\{(v,w) \in P^r \oplus P^{m s} \mid \psi(v) =
    \eta^{\oplus s}(w)\bigr\} \subseteq P^r \oplus P^{m s}.
    \]
    With $\map{\pi_i}{P^r \oplus P^{m s}}{P^r,P^{m s}}$ the
    projections, the diagram%
    \vspace{-5mm}%
    \DIAGV{60}%
    {} \n{} \n{} \n{} \n{R^G} \nn%
    {} \n{} \n{} \n{} \n{\sar} \nn%
    {M} \n{\Ear{\pi_1}} \n{P^r} \n{\Ear{\phi}} \n{R} \n{\ear} \n{0}
    \nn%
    {\Sar{\pi_2}} \n{} \n{\Sar{\psi}} \n{} \n{\saR{\underline{\sigma}
        - \underline{\id}}} \nn%
    {P^{m s}} \n{\Ear{\eta^{\oplus s}}} \n{P^{r s}}
    \n{\Ear{\phi^{\oplus s}}} \n{R^s} \n{\ear} \n{0}%
    \diag%
    of $P$-modules commutes and has the last row and the third column
    exact (but the second last row only at $R$). It follows by a
    diagram chase and from the definition of $M$ that the
    $\phi(\pi_1(m_j))$ generate $R^G$ as a module over $A$. So the
    $b_i$ and the $\phi(\pi_1(m_j))$ generate $R^G$ as a $K$-algebra.
  \end{enumerate}
\end{alg}

We make comments on the steps of the algorithm.
\begin{enumerate}
\item[\eqref{aNondomain1}] The $b_i$ can simply be taken as the
  coefficients of the polynomials in~\eqref{eqNoether}. However, the
  efficiency of the algorithm hinges on minimizing~$m$ and~$r$, so
  choosing a larger subalgebra $A$ will be beneficial. After
  constructing a $K$-domain $S$ with a $G$-action and a
  $G$-equivariant, surjective homomorphism $S \to R$, one can apply
  \aref{aFinite} to $S$ and then take $A$ as the image of $S^G$ in
  $R$. For example, $S$ can be chosen as a polynomial ring with
  variables $x_{i,\sigma}$ mapping to $\sigma \cdot a_i$ (where
  $\sigma \in G$, $i = 1 \upto n$) with the obvious $G$-action.
\item[\eqref{aNondomain2}] Since by~\eqref{eqNoether} the $a_i$
  satisfy integral equations over $A$ of degree~$|G|$, the $c_j$ can
  be chosen as the products $\prod_{i=1}^n a_i^{e_i}$ with $0 \le e_i
  < |G|$. But better choices may be beneficial.
\item[\eqref{aNondomain3}] An algorithm for performing this step is
  given by \lref{lKer} below.
\item[\eqref{aNondomain4}] To obtain the $\psi$-image of the $j$th
  free generator of $P^r$, one needs to compute a $\phi$-preimage of
  $\sigma_i(c_j) - c_j$ for $i = 1 \upto s$. So the $\sigma_i(c_j) \in
  R$ need to be represented as $P$-linear combinations of the
  $c_k$. This can be done by applying \lref{lKer} to $c_1 \upto
  c_r,\sigma \cdot c_j$ and using the last statement of the lemma.
\item[\eqref{aNondomain5}] $M$ is the kernel of the map
  \[
  \map{- \psi \oplus \eta^{\oplus s}}{P^r \oplus P^{m s}}{P^{r s}}.
  \]
  Algorithms for computing the kernel of a $P$-linear map of free
  $P$-modules of finite rank are well-known (see
  \mycite[Exercise~4.3.15d]{AL}).
\end{enumerate}

Steps~\ref{aNondomain3} and~\ref{aNondomain4} of the algorithm require
the following lemma.

\begin{lemma} \label{lKer}%
  Let $I \subseteq K[x_1 \upto x_n]$ be an ideal in a polynomial ring
  over a ring $K$, and let $f_1 \upto f_m,h_1 \upto h_r \in K[x_1
  \upto x_n]$ be polynomials with $h_1 = 1$ defining a map
  \[
  \mapl{\phi}{K[y_1 \upto y_m]^r}{K[x_1 \upto x_n]/I := R}{(g_1 \upto
    g_r)}{\sum_{j=1}^r g_j(f_1 \upto f_m) h_j + I},
  \]
  where the $y_i$ are indeterminates. Let $J \subseteq K[x_1 \upto
  x_n,y_1 \upto y_m,z_2 \upto z_r]$ be the ideal generated by $I$ and
  by all $y_i - f_i$ and $z_j - h_j$ (with $z_j$ further
  indeterminates). Choose a monomial ordering on
  $K[\underline{x},\underline{y},\underline{z}]$ such that every $x_i$
  is bigger than every monomial in the $y$-and $z$-variables, and such
  that if $\deg_z(s) < \deg_z(t)$ for two monomials $s,t \in
  K[\underline{y},\underline{z}]$, then $s < t$. Let $\mathcal{G}$ be
  a Gr\"obner basis of $J$ with respect to this ordering. Write
  $K[\underline{y},\underline{z}]_{\le 1} := \bigl\{f \in
  K[\underline{y},\underline{z}] \mid \deg_z(f) \le 1\bigr\}$, and set
  \[
  \mathcal{G}' := \left(K[\underline{y},\underline{z}]_{\le 1} \cap
    \mathcal{G}\right) \cup \bigl\{z_j g \mid g \in K[\underline{y}]
  \cap \mathcal{G}, \ j = 2 \upto n\bigr\} \subseteq
  K[\underline{y},\underline{z}]_{\le 1}.
  \]
  For $g = g_1 + g_2 z_2 + \cdots g_r z_r \in
  K[\underline{y},\underline{z}]_{\le 1}$ with $g_i \in
  K[\underline{y}]$, write $\overrightarrow{g} =(g_1 \upto g_r)$. Then
  the kernel of~$\phi$ is generated (as a $K[\underline{y}]$-module)
  by $\{\overrightarrow{g} \mid g \in \mathcal{G}'\}$. Moreover, if
  $z_r > z_j t$ for $j < r$ and~$t$ a monomial in the $y$-variables,
  and if there is a vector in $\ker(\phi)$ whose last component
  is~$1$, then such a vector can be obtained as a $K$-linear
  combination of the $\overrightarrow{g}$ with $g \in \mathcal{G}'$.
\end{lemma}

\begin{proof}
  Let $M \subseteq K[\underline{y},\underline{z}]_{\le 1}$ be the
  $K[\underline{y}]$-submodule generated by $\mathcal{G}'$. With
  \[
  \map{\widehat{\phi}}{K[\underline{y},\underline{z}]_{\le 1}}{R}, \
  y_i \mapsto f_i + I, \ z_j \mapsto h_j + I,
  \]
  we need to show $\ker(\widehat{\phi}) = M$. Under the additional
  hypothesis on the monomial ordering we also need to show that if
  there exists an element in $\ker(\widehat{\phi})$ with~$1$ as the
  coefficient of $z_r$, then such an element occurs as a $K$-linear
  combination of $\mathcal{G}'$.

  We have $\mathcal{G}' \subseteq K[\underline{y},\underline{z}]_{\le
    1} \cap J$. Moreover, if $g \in K[\underline{y},\underline{z}]
  \cap J$, then $g(f_1 \upto f_m,h_2 \upto h_r) \in I$, so
  $K[\underline{y},\underline{z}]_{\le 1} \cap J \subseteq
  \ker(\widehat{\phi})$. It follows that $M \subseteq
  \ker(\widehat{\phi})$.
  
  To prove the reverse inclusion, take $g \in
  \ker(\widehat{\phi})$. Then $g(f_1 \upto f_m,h_2 \upto h_r) \in I
  \subseteq J$ and $g - g(f_1 \upto f_m,h_2 \upto h_r) \in J$, so $g
  \in J$. By the definition of a Gr\"obner basis (over a ring) it
  follows that the leading term $\LT(g)$ lies in the ideal generated
  by the leading terms of the elements of $\mathcal{G}$. So
  \begin{equation} \label{eqLT}%
    \LT(g) = \sum_{i=1}^s c_i t_i \LT(g_i)
  \end{equation}
  with $g_1 \upto g_s \in \mathcal{G}$, $c_i \in K$, and $t_i$
  monomials such that $t_i \LM(g_i) = \LM(g)$. Since $g \in
  K[\underline{y},\underline{z}]_{\le 1}$, the same follows for
  the~$t_i$ and $\LT(g_i)$. By the properties of the monomial
  ordering, this implies $g_i \in \mathcal{G}'$. Suppose that the
  $\deg_z(t_i) = 1$ for an~$i$. Then $\deg_z\bigl(\LM(g_i)\bigr) = 0$,
  so $g_i \in K[\underline{y}] \cap \mathcal{G}$. With $t_i = z_{j_i}
  \widehat{t_i}$ and $\widehat{g_i} := z_{j_i} g_i \in \mathcal{G}'$
  we have $c_i t_i \LT(g_i) = c_i \widehat{t_i} \LT(\widehat{g_i})$
  and $\widehat{t_i} \in K[\underline{y}]$. On the other hand, if
  $\deg_z(t_i) = 0$, then the same properties hold with $\widehat{t_i}
  := t_i$ and $\widehat{g_i} := g_i$. It follows that $\widehat{g} :=
  \sum_{i=1}^s c_i \widehat{t_i} \widehat{g_i} \in M$ and $\LM(g -
  \widehat{g}) < \LM(g)$. So if we assume $M \subsetneqq
  \ker(\widehat{\phi})$ and choose $g \in \ker(\widehat{\phi})
  \setminus M$ with $\LM(g)$ minimal, we arrive at a contradiction.

  To prove the last assertion, assume that the $z_r$-coefficient of $g
  \in \ker(\widehat{\phi})$ is~$1$. Then $\LT(g) = z_r$ by the
  additional hypothesis on the monomial ordering. So the equations
  $\widehat{t_i} \LM(\widehat{g_i}) = t_i \LM(g_i) = \LM(g)$ together
  with $\widehat{t_i}$ in $K[\underline{y}]$ imply $\widehat{t_i} =
  1$. It follows by~\eqref{eqLT} that $\widehat{g} = \sum_{i=1}^s c_i
  \widehat{g_i}$ has $z_r$-coefficient~$1$, as claimed.
\end{proof}

We finish the section with an example.

\begin{ex} \label{exNondomain}%
  The cyclic group $G$ of order~$2$ acts on $R := \ZZ[x_1,x_2]/(x_1^2
  - x_2^2)$ by exchanging $\overline{x}_1$ and $\overline{x}_2$ (where
  the bars indicate classes modulo $(x_1^2 - x_2^2)$). We apply
  \aref{aNondomain}. In step~\ref{aNondomain1} we can choose $A =
  \ZZ[f_1,f_2]$ with
  \[
  f_1 = \overline{x}_1 + \overline{x}_2 \quad \text{and} \quad f_2 =
  \overline{x}_1 \overline{x}_2,
  \]
  which is also the image of $\ZZ[x_1,x_2]^G \to R^G$. In
  step~\ref{aNondomain2} we can choose $c_1 = 1$ and $c_2 =
  \overline{x}_1$. The computation of $\ker(\phi)$ in
  step~\ref{aNondomain3} is too hard to do by hand. The result is a
  submodule of $P^2$ generated by $m = 4$ elements, where $P =
  \ZZ[y_1,y_2]$. The computation of $M$ amounts to computing the
  kernel of a map $P^6 \to P^2$. The computation, done with MAGMA,
  produces~$4$ generators of $M$, but only one of them yields a new
  invariant. This invariant is $f_3 = \overline{x}_1^2$, so the result
  of the computation is $R^G = \ZZ[f_1,f_2,f_3]$. We have $2 f_3 =
  f_1^2 - 2 f_2$, so $f_3$ would lie in $A$ if we were computing over
  a ring in which~$2$ is invertible. But over $\ZZ$ we see that
  although the map $\ZZ[x_1,x_2] \to R$ is $G$-equivariant and
  surjective, its restriction to invariants is not surjective.

  The total computation time for this example was 0.01 seconds.
\end{ex}

For the rest of the paper we will focus on the case of infinite
groups.

\section{The Italian problem} \label{sItalian}

We now discuss the question, raised by \tref{tLocalize}, whether
there exists $a \in R^G$ such that $A \subseteq R_a$. The discussion
was postponed until now for not disturbing the flow of ideas.

\begin{prop} \label{pRa}%
  Assume the situation given by Assumption~\ref{Assumption}. If
  \[
  L^G = \Quot(R^G),
  \]
  then there exists a nonzero invariant $a \in R^G$ such that $A
  \subseteq R_a$. If $L^G = \Quot(A)$ (which by
  \tref{tInvariantField} is guaranteed to hold if $E$ is tamely
  extended), then the converse holds.
\end{prop}

\begin{proof}
  Suppose $L^G = \Quot(R^G)$. Since $A$ is a finitely generated
  subalgebra of $L^G$, we can choose $a \in R^G \setminus \{0\}$ as a
  common denominator of the generators. The $A \subseteq
  R_a$. Conversely, $A \subseteq R_a$ with $a \in R^G \setminus \{0\}$
  implies
  \[
  A \subseteq R_a \cap L^G = (R_a)^G = (R^G)_a\subseteq \Quot(R^G),
  \]
  so $L^G = \Quot(A) \subseteq \Quot(R^G) \subseteq L^G$.
\end{proof}

In the standard situation of invariant theory where $G$ is a linear
algebraic group, $X$ an irreducible $G$-variety and $R = K[X]$, the
above proposition raises the question whether the invariant field
$K(X)^G$ coincides with the field of fractions of the invariant ring
$K[X]^G$. This question is sometimes referred to as the {\em Italian
  problem} (see \mycite[page~183]{Mukai:2003}). A typical example
where $K(X)^G$ and $\Quot\left(K[X]^G\right)$ are different is the
action of the multiplicative group $\Gm$ on $K^2$ with weight $(1,1)$
(see \exref{exGm}). This example also shows that for nontamely
extended Derksen ideals the converse in \pref{pRa} may fail. In other
words, nontamely extended Derksen ideals can serve to compute (a
localization of) $K[X]^G$ even if the Italian problem has a negative
answer.

We will give a positive answer to the Italian problem in two cases. As
a preparation we need the following proposition, which was proved (in
a more general situation) by
\mycite[Lemma~4.13]{Hashimoto:factorial}. For the convenience of the
reader we include a proof here that is adapted to our
situation. Notice that the proposition would be almost trivial (and
would not require the hypothesis that $G$ is connected) with the
additional hypothesis that $K[X]$ has $K \setminus \{0\}$ as its group
of units.

\begin{prop}[Hashimoto] \label{pHashimoto}%
  Let $G$ be a connected linear algebraic group over an algebraically
  closed field $K$, $X$ an irreducible $G$-variety, and $f \in
  K[X]$. If the ideal $(f) \subseteq K[X]$ is $G$-stable, then there
  exists a homomorphism (i.e., a group homomorphism that is a morphism
  of varieties) $\map{\chi}{G}{\Gm}$ such that $\sigma \cdot f =
  \chi(\sigma) f$ for all $\sigma \in G$.
\end{prop}

\begin{proof}
  We may clearly assume $f \ne 0$.

  Embedding $X$ into a $G$-module $V$ (see
  \mycite[Algorithm~1.2]{Derksen.Kemper06}) yields a $G$-equivariant
  epimorphism $K[x_1 \upto x_n] = K[V] \to K[X]$ from a polynomial
  ring onto $K[X]$. Let $I$ be its kernel and $I_\proj \subseteq K[x_0
  \upto x_n]$ the homogenization. With $G$ fixing $x_0$, it is
  straightforward to check that $I_\proj$ is $G$-stable. So $G$ acts
  on the projective variety $X_\proj \subseteq \PP^n$ given by
  $I_\proj$, with the action given by a morphism $G \times X_\proj \to
  X_\proj$. Since $X_\proj$ is the projective closure of $X$, we have
  an injective, birational and $G$-equivariant morphism $X
  \hookrightarrow X_\proj$.

  \newcommand{\Y}{{\widetilde{X}_\proj}}%
  Let $\Y \to X_\proj$ be the normalization of $X_\proj$. This is a
  birational, finite morphism, so in particular it is proper (see
  \mycite[II, Exercise~4.1]{hart}). It follows that $\Y$ is a complete
  variety, which implies
  \begin{equation} \label{eqGamma}
    \Gamma\bigl(\Y,\mathcal{O}_\Y\bigr) = K
  \end{equation}
  (see \mycite[II, Exercise~4.5]{hart}). For an irreducible closed
  subset $Z \subseteq \Y$, we view the local ring $\mathcal{O}_{\Y,Z}$
  as a subring of the function field $K\bigl(\Y\bigr)$ and claim that
  \begin{equation} \label{eqIntersect}
    \Gamma\bigl(\Y,\mathcal{O}_\Y\bigr) = \bigcap_{\codim(Z) = 1}
    \mathcal{O}_{\Y,Z}.
  \end{equation}
  Clearly the left hand side is contained in the right hand side. For
  the converse, let $g \in K\bigl(\Y\bigr)$ be a rational function
  whose domain of definition $U_g$ is strictly contained in $\Y$. Then
  there exists an affine open subset $X' \subseteq \Y$ that is not
  contained in $U_g$. Since $X'$ is normal, it follows by
  \mycite[Theorem~38]{Matsumura:80} that there exists an closed
  irreducible subset $Z' \subset X'$ of codimension~$1$ that does not
  meet $U_g$. Since $U_g$ is open, it follows that the closure $Z :=
  \overline{Z}'$ in $\Y$ does not meet $U_g$. By \lref{lTop} (which
  is proved below), $Z$ is irreducible of codimension~$1$. So~$h$ is
  not contained in $\mathcal{O}_{\Y,Z}$, and~\eqref{eqIntersect} is
  proved. Together with~\eqref{eqGamma}, we obtain
  \begin{equation} \label{eqInter}%
    \bigcap_{\codim(Z) = 1} \mathcal{O}_{\Y,Z} = K.
  \end{equation}
  Let $\sigma \in G$. The composition $\Y \to X_\proj
  \xrightarrow{\sigma} X_\proj$ is dominant, so by the universal
  property of the normalization it factors uniquely through $\Y$, and
  we obtain a morphism $\map{\sigma}{\Y}{\Y}$. This defines a
  $G$-action on $\Y$ such that the map $\Y \to X_\proj$ is
  $G$-equivariant. To see that the action is given by a morphism $G
  \times \Y \to \Y$, we remark that $G \times \Y$ is normal since it
  follows from \mycite[Proposition~6.14.1]{EGA42} that a product of
  normal varieties over an algebraically closed field is normal. It
  follows that the composition $G \times \Y \to G \times X_\proj \to
  X_\proj$ factors uniquely through $\Y$, giving a morphism
  $\map{\phi}{G \times \Y}{\Y}$. For $\sigma \in G$ we have the
  commutative diagram%
  \DIAGV{100}%
  {\Y} \n{\Ear{x \mapsto (\sigma,x)}} \n{G \times \Y} \n{\Ear{\phi}}
  \n{\Y} \nn%
  {\sar} \n{} \n{\sar} \n{} \n{\sar} \nn%
  {X_\proj} \n{\Ear{x \mapsto (\sigma,x)}} \n{G \times X_\proj}
  \n{\ear} \n{X_\proj}%
  \diag%
  This shows that $\phi(\sigma,x) = \sigma \cdot x$ for all $x \in
  \Y$, as claimed.

  The normal locus $X_\norm \subseteq X$ is $G$-stable and open (see
  \mycite[Proposition~6.13.2]{EGA42}), and the composition $X_\norm
  \hookrightarrow X \hookrightarrow X_\proj$ factors uniquely through
  $\Y$. This gives a $G$-equivariant, injective and birational
  morphism $X_\norm \to \Y$. By Zariski's Main Theorem (see
  \mycite[page~209]{Mumford:1999}), it is an isomorphism of $X_\norm$
  with an open subset of $\Y$, so we can identify $X_\norm$ with an
  open subset of $\Y$.

  After these preparations we turn our attention to the regular
  function $f \in K[X]$. For $\sigma \in G$, we claim that
  $\frac{\sigma \cdot f}{f} \in K$. Since~$f$ is defined on $X_\norm
  \subseteq \Y$, it gives rise to a rational function on $\Y$, which
  we also write as~$f$. By~\eqref{eqInter} we need to show that if $Z
  \subset \Y$ is an irreducible closed subset of codimension~$1$, then
  $\frac{\sigma \cdot f}{f} \in \mathcal{O}_{\Y,Z}$. First assume that
  $X_\norm \cap Z = \emptyset$. Then $Z$ is an irreducible component
  of $\Y \setminus X_\norm$. Since $G$ acts morphically on $\Y
  \setminus X_\norm$, it follows (using the connectedness of $G$) that
  $Z$ is $G$-stable. This prompts us to drop the assumption that
  $X_\norm \cap Z = \emptyset$ and treat the more general case that
  $Z$ is $G$-stable. In this case $G$ acts on $\mathcal{O}_{\Y,Z}$,
  which is a discrete valuation ring. The action is by automorphisms
  and therefore preserves the valuation, so $\frac{\sigma \cdot f}{f}
  \in \mathcal{O}_{\Y,Z}^\times$.

  Now assume that $Z$ is not $G$-stable, which implies $Z' := X_\norm
  \cap Z \ne \emptyset$. Then $Z'$ is not $G$-stable, either, since
  otherwise $\overline{Z'} \subseteq \Y$, which by \lref{lTop} equals
  $Z$, would be $G$-stable. By hypothesis, the set $Y := \{x \in
  X_\norm \mid f(x) = 0\} \subsetneqq X_\norm$ is $G$-stable, so the
  same is true for its irreducible components. So $Z'$ is not an
  irreducible component of $Y$. Since by \lref{lTop}, $Z'$ is
  irreducible of codimension~$1$, this implies $Z' \not\subseteq
  Y$. But this means that $\frac{\sigma \cdot f}{f}$ is defined at
  some point from $Z' \subseteq Z$, so it lies in
  $\mathcal{O}_{\Y,Z}$.

  We conclude that $\frac{\sigma \cdot f}{f} \in K$, as claimed. This
  means that $K \cdot f \subseteq K[X]$ is $G$-stable. Since the
  $G$-action on $K[X]$ is locally finite (see \mycite[Chapter~1, \S~1,
  Lemma]{MFK}), it follows that $K \cdot f$ is a $G$-stable subset of
  a $G$-module and therefore itself a $G$-module. So there exists a
  homomorphism $\map{\chi}{G}{\GL_1 = \Gm}$ as claimed.
\end{proof}

The following lemma was used in the above proof.

\begin{lemma} \label{lTop}%
  Let $Y \subseteq X$ be an open subset in a topological
  space. Sending an irreducible closed subset $Z \subseteq Y$ to its
  closure $\overline{Z}$ in $X$ provides an inclusion-preserving
  bijection between the irreducible closed subsets of $Y$ and those
  irreducible closed subsets of $X$ that meet $Y$. The inverse
  bijection is given by intersecting with $Y$.
\end{lemma}

\begin{proof}
  The proof is straightforward and is left to the reader.
  %
\end{proof}

\begin{rem*}
  The hypothesis that $G$ be connected cannot be dropped from
  \pref{pHashimoto}. For example, if $X \subseteq K^2$ is given by
  the equation $x y = 1$ and $G \cong C_2$ acts be exchanging the
  coordinates, then the assertion of the proposition fails for $f =
  x$, even though $(f) = K[X]$ is $G$-stable.
\end{rem*}

We now come to the announced result on the Italian problem. The first
part of the following theorem is folklore (see, for example,
\mycite{Kamke:Diss,Kamke:Kemper:2011}), and the second is a special
case of \mycite[Proposition~5.1]{Hashimoto:factorial}.

\begin{theorem} \label{tItalian}%
  Let $G$ be a linear algebraic group over an algebraically closed
  field $K$ and $X$ an irreducible $G$-variety. Then the equality
  $K(X)^G = \Quot\left(K[X]^G\right)$ holds if
  \begin{enumerate}
  \item \label{tItalianA} the identity component $G^0$ is unipotent,
    or
  \item \label{tItalianB} $K[X]$ is a unique factorization domain and
    every homomorphism $G^0 \to \Gm$ to the multiplicative group is
    trivial.
  \end{enumerate}
\end{theorem}

\begin{proof}
  \begin{enumerate}
  \item[\eqref{tItalianA}] Let $a \in K(X)^G$. The set $J := \{d \in
    K[X] \mid d a \in K[X]\} \subseteq K[X]$ is a nonzero, $G$-stable
    ideal. Since the $G$-action on $K[X]$ is locally finite, $J$
    contains a nonzero $G$-module $V$.  Since $G^0$ is unipotent,
    $V^{G^0} \ne \{0\}$. For $c \in V^{G^0} \setminus \{0\}$, the
    product $\prod_{\sigma \in G/G^0} \sigma \cdot c$ is nonzero and
    lies in $K[X]^G \cap J$. This implies that $a \in
    \Quot\left(K[X]^G\right)$.
  \item[\eqref{tItalianB}] Let $a = b/d \in K(X)^G$ with $b,d \in
    K[X]$, which we may assume to be coprime. Then for every $\sigma
    \in G$, the equation $(\sigma \cdot b) \cdot d = b \cdot (\sigma
    \cdot d)$ implies that $\sigma \cdot d$ lies in the ideal $(d)
    \subseteq K[X]$, so this ideal is $G$-stable. Hence by
    \pref{pHashimoto}, $\sigma \cdot d = \chi(\sigma) d$ for $\sigma
    \in G^0$ with $\map{\chi}{G^0}{\Gm}$ a homomorphism. By
    hypothesis, $\chi$ is trivial, so $d \in K[X]^{G^0}$. Now the
    product $\prod_{\sigma \in G/G^0} (\sigma \cdot d)$ is a nonzero
    element of $K[X]^G \cap J$, with $J$ as above. This shows that $a
    \in \Quot\left(K[X]^G\right)$. \qed
  \end{enumerate}
  \renewcommand{\qed}{}
\end{proof}

Notice that if the $G^0$ is a perfect group, then $G$ satisfies the
last assumption of \tref{tItalian}\eqref{tItalianB}. For example,
this holds for the special linear groups, the orthogonal groups, and
the special orthogonal groups.

\begin{rem*}
  Example~3.15 from \mycite{Kamke:Diss} shows that the hypothesis that
  $K[X]$ is factorial cannot be replaced by the weaker hypothesis that
  $K[X]$ is normal.
\end{rem*}

\section{Extended Derksen ideals: geometric
  aspects} \label{sDerksenG}%

In order to give a geometric interpretation to an (extended) Derksen
ideal, we assume that $G$ is a linear algebraic group over an
algebraically closed field $K$ and $S = K[X]$ is the coordinate ring
of a $G$-variety $X$. We consider (extended) Derksen ideals with
respect to some choice of elements $a_1 \upto a_n \in S$. For
simplicity we assume that the~$a_i$ generate $S$, so $R = K[a_1 \upto
a_n] = S$. Then the~$a_i$ define an injective and closed morphism
\[
\mapl{\phi}{X}{K^n}{v}{\bigl(a_1(v) \upto a_n(v)\bigr)}.
\]
By definition, the Derksen ideal $D_{a_1 \upto a_n}$ is the
intersection of the radical ideals belonging to the closed subsets
$\bigr\{(\sigma \cdot x,\phi(x)) \mid x \in X\bigr\} \subseteq X
\times K^n$ (for all $\sigma \in G$), and so $D_{a_1 \upto a_n}$ is
the radical ideal belonging to the closure $\overline{\mathcal{D}}$ of
the set
\[
\mathcal{D} := \bigr\{(\sigma \cdot x,\phi(x)) \mid x \in X, \ \sigma
\in G\bigr\} \subseteq X \times K^n.
\]
Of course if the $a_i$ are indeterminates of a polynomial ring, then
$\mathcal{D}$ is just the ``graph of the action.'' Let us consider a
tamely extended Derksen ideal $E \subseteq R[y_1 \upto y_n]$. This
defines a closed subset $\mathcal{E} \subseteq X \times K^n$. With
$\map{\pi_{K^n}}{X \times K^n}{K^n}$ the second projection, the ideal
$I \subseteq R$ from \dref{dDerksen}\eqref{dDerksenC} defines the
closed subset
\begin{equation} \label{eqZ}%
  Z := \phi^{-1}\bigl(\overline{\pi_{K^n}(\mathcal{E})}\bigr)
  \subseteq X,
\end{equation}
and the condition that $\bigcap_{\sigma \in G} \sigma \cdot I$ be
nilpotent is equivalent to the condition that the set $G \cdot Z$ of
points from $X$ whose orbit meets $Z$ is dense in $X$.

Conversely, let $Z \subseteq X$ be a closed subset such that $G \cdot
Z$ is dense in $X$, and define the set
\begin{equation} \label{eqE}%
  \mathcal{E}_Z := \bigl\{(\sigma \cdot z,\phi(z)) \mid z \in Z, \
  \sigma \in G\bigr\} \subseteq X \times K^n,
\end{equation}
which need not be closed. It is easy to see that the vanishing ideal
$E_Z \subseteq R[y_1 \upto y_n]$ of $\mathcal{E}_Z$ is a tamely
extended Derksen ideal. In fact, forming the closed subset of $X$
belonging to $\mathcal{E}_Z$ as in~\eqref{eqZ} yields the set $Z$
with which we have started. As we will see, starting with a closed
subset $Z \subseteq X$ with $G \cdot Z$ dense has the additional
benefit of producing a tamely extended Derksen ideal when one works
with $S = \Quot(R) = K(X)$ instead of $S = R$ (see
\tref{tComputeDE}\eqref{tComputeDEB} and \rref{rComputeDE}), so one
is in the situation of Assumption~\ref{Assumption}.

In summary, tamely extended Derksen ideals $E$ are intimately related
to closed subsets $Z \subseteq X$ such that $G \cdot Z$ (the set of
points whose orbit meets $Z$) is dense in $X$. Using this relationship
and passing from a $Z$ to the corresponding $E$ and back yields the
original $Z$. However, passing from an $E$ to the corresponding $Z$
and back will in general not yield the original $E$. For example, if
the multiplicative group $\Gm$ acts on $X = K^2$ with weight $(1,1)$,
then $D_{x_1,x_2} = (x_1 y_2 - x_2 y_1)$, and $E = (x_1,x_2)$ is a
tamely extended Derksen ideal, for which $E_Z = D_{x_1,x_2}$.

Our condition on $Z$ is related to the concept of cross-sections from
\mycite[Section~3.1]{HK:06}, which in fact motivated the first
definition, made in \mycite{Kamke:Kemper:2011}, of extended Derksen
ideals. However, cross-sections in the sense of Hubert and Kogan are
more restrictive, since they require that a generic orbit meets the
cross-section in only finitely many points. We do not impose this
finiteness condition.

Still further away from the notions of Hubert and Kogan are nontamely
extended Derksen ideals, which we consider now. In fact, we will be
interested in the condition $R \cap E = \{0\}$ (which is stronger than
$E \subsetneqq R[y_1 \upto y_n]$), since this will produce an extended
Derksen ideal also when working with $S = K(X)$ instead of $S =
K[X]$. The following result gives a geometric method for constructing
such extended Derksen ideals.

\begin{prop} \label{pWild}%
  In above situation, let $(f_1 \upto f_s) \subseteq K[y_1 \upto y_n]$
  be an ideal and consider the subvariety $Z \subseteq X$ of points
  vanishing at all $f_i(a_1 \upto a_n)$. Suppose that the set
  \[
  \mathcal{M} := \bigl\{x \in X \mid \overline{G \cdot x} \cap Z \ne
  \emptyset\bigr\}
  \]
  of points whose orbit closure passes through $Z$ is dense in
  $X$. Then $E := D_{a_1 \upto a_n} + (f_1 \upto f_s)$ is an extended
  Derksen ideal satisfying $R \cap E = \{0\}$.
\end{prop}

\begin{proof}
  It is clear that $E$ is $G$-stable and contains $D_{a_1 \upto a_n}$,
  so we only need to show that $R \cap E = \{0\}$. Let $x \in
  \mathcal{M}$. Then there exists $z \in \overline{G \cdot x} \cap Z$,
  so all~$f_i$ vanish at $\bigl(x,\phi(z)\bigr) \in X \times
  K^n$. Take $h \in D_{a_1 \upto a_n}$. Then $h\bigl(x,\phi(\sigma
  \cdot x)\bigr) = 0$ for all $\sigma \in G$, so the function $X \to
  K$, $y \mapsto h\bigl(x,\phi(y)\bigr)$ vanishes on $G \cdot
  x$. Since it is continuous, it also vanishes on $\overline{G \cdot
    x}$. In particular, $h\bigl(x,\phi(z)\bigr) = 0$. We conclude that
  $g\bigl(x,\phi(z)\bigr) = 0$ for all $g \in E$. In particular, if $g
  \in R \cap E$, then $g(x) = 0$. Since~$x$ was taken as an arbitrary
  element of $\mathcal{M}$, this implies $g = 0$.
\end{proof}

\begin{rem*}
  My attempts to prove the following converse were unsuccessful: If an
  ideal $E = D_{a_1 \upto a_n} + (f_1 \upto f_s)$ as in \pref{pWild}
  is an extended Derksen ideal with $R \cap E = \{0\}$, then the set
  $\mathcal{M}$ is dense in $X$.
\end{rem*}

In the following example we see an extended Derksen ideal that is not
tame. The example also shows that the hypothesis that $E$ be tamely
extended cannot be dropped from \tref{tInvariantField}.

\begin{ex} \label{exGm}%
  Consider the action of the multiplicative group $\Gm$ on $X = K^2$
  with weight $(1,1)$. The Derksen ideal with respect to the
  indeterminates~$x_1$ and~$x_2$ is $D_{x_1,x_2} = (x_1 y_2 - x_2
  y_1)$. Since all orbit closures contain the origin, \pref{pWild}
  yields that $E = (y_1,y_2)$ is an extended Derksen ideal with $R
  \cap E = \{0\}$. With $L = K(X) = K(x_1,x_2)$, the ideal in
  $L[y_1,y_2]$ generated by~$y_1$ and~$y_2$ is also an extended
  Derksen ideal. We already have a Gr\"obner basis, and
  \tref{tLocalize} tells us that $R^G = K$. This argument always
  applies when all orbit closures meet in one point, and yields the
  well-known result that in such a situation no nonconstant invariants
  exist.
  
  Trying to apply \tref{tInvariantField} would yield $L^G = K$, which
  is incorrect. This implies that $E$ is not tamely extended, which
  can also be seen directly.

  For computing $L^G$ we could use the set $Z \subseteq X$ given by
  the equation $x_1 = 1$. From this we get the tamely extended Derksen
  ideal $(y_1 - 1,x_1 y_2 - x_2) \subseteq L[y_1,y_2] $ with Gr\"obner
  basis $\mathcal{G} = \{y_1 - 1,y_2 - (x_2/x_1)\}$. By
  \tref{tInvariantField}, $L^G = K(x_2/x_1)$, which is correct.
\end{ex}

Notice that this example could not be treated by the methods of
\mycite{HK:06}. The example gives some hints of the usefulness of
extended Derksen ideals as a generalization of Derksen ideals. Their
(potential) benefit is threefold: (1) They may reduce the cost of the
Gr\"obner basis computation, (2) they may reduce the number of
coefficients occurring in the Gr\"obner basis, and (3) they may help
to achieve that $a \in R^G$ exists with $A \subseteq R_a$ (using the
notation of \tref{tLocalize}), or even that $a = 1$. In particular~(3)
is nicely illustrated by \exref{exGm}. It is not clear to me how far
one can get with this: For which groups $G$ and $G$-varieties $X$ does
there exist a nontamely extended Derksen ideal such that one can
achieve~\eqref{eqLocalize} in \tref{tLocalize}?

It will be important to determine the Krull dimension of tamely
extended Derksen ideals.

\begin{lemma} \label{lDimDerksen}%
  Let $G$ be a linear algebraic group, $X$ a $G$-variety, and $a_1
  \upto a_n$ generators of $R := K[X]$. Moreover, let $E = E_Z
  \subseteq R[y_1 \upto y_n]$ be a tamely extended Derksen ideal
  formed as in~\eqref{eqE} from a closed subset $Z \subseteq X$ with
  $G \cdot Z$ dense. If $Z$ is equidimensional (i.e., all irreducible
  components have the same dimension), then
  \[
  \dim\left(R[y_1 \upto y_n]/E\right) = \dim(Z) + d,
  \]
  where~$d$ is the maximal dimension of a $G$-orbit in $X$ (which is
  attained on a nonempty open subset of $X$, see
  \mycite[Section~10.3]{Kemper.Comalg}). If $Z$ is not equidimensional,
  the right hand side of the above equation is an upper bound. Notice
  that in the case $Z = X$ we have $E = D_{a_1 \upto a_n}$.
\end{lemma}

\begin{proof}
  The set $\mathcal{E}_Z$ from~\eqref{eqE} is the image of the
  morphism
  \[
  \mapl{\psi}{G \times Z}{X \times K^n}{(\sigma,z)}{\bigl(\sigma \cdot
    z,\phi(z)\bigr)}.
  \]
  We need to determine the dimension of
  $\overline{\mathcal{E}_Z}$. Let $\sigma_1 G^\circ \upto \sigma_m
  G^\circ$ be the connected components of $G$. Then
  $\overline{\mathcal{E}_Z}$ is the union of the
  $\overline{\psi(\sigma_i G^\circ \times Z)}$, which are all
  isomorphic to each other. So
  $\dim\left(\overline{\mathcal{E}}\right) =
  \dim\bigl(\overline{\psi(G^\circ \times Y)}\bigr)$. Since~$d$ does
  not change when substituting $G$ by $G^\circ$, we may assume $G$ to
  be connected.

  Let $(x,v) \in X \times K^n$ be in the image of~$\psi$, so $x =
  \sigma \cdot z$ and $v = \phi(z)$ with $\sigma \in G$ and $z \in
  Y$. Since~$\phi$ is injective by assumption, the fiber of $(x,v)$ is
  \[
  \psi^{-1}\bigl(\{(x,v)\}\bigr) = G_{x} \sigma \times \{z\} \cong
  G_x.
  \]
  
  Let $Z_1 \upto Z_r$ be the irreducible components of $Z$ and set
  $\mathcal{E}_i := \psi(G \times Z_i)$. Then
  \[
  \dim\bigl(\overline{\mathcal{E}_Z}\bigr) =
  \max\bigl\{\dim\bigl(\overline{\mathcal{E}_1}\bigr) \upto
  \dim\bigl(\overline{\mathcal{E}_r}\bigr)\bigr\}.
  \]
  Since the $\overline{\mathcal{E}_i}$ are irreducible, a standard
  result about the dimensions of fibers (see
  \mycite[Corollary~10.6]{Kemper.Comalg}) tells us that there exist
  nonempty open subsets $U_i \subseteq \overline{\mathcal{E}_i}$ such
  that for every $(x,v) \in U_i$ the equation
  \begin{equation} \label{eqDimGv}%
    \dim(G_x) = \dim(G) + \dim(Z_i) -
    \dim\bigl(\overline{\mathcal{E}_i}\bigr)
  \end{equation}
  holds. Since $\dim\bigl(G \cdot x\bigl) = \dim(G) - \dim(G_x)$, the
  upper bound $\dim\bigl(\overline{\mathcal{E}_i}\bigr) \le \dim(Z_i)
  + d \le \dim(Z) + d$ follows.

  Set $U := U_1 \cup \cdots \cup U_r$ and consider the projection
  $\map{\pi}{X \times K^n}{X}$. Since $\overline{U} =
  \overline{\mathcal{E}_Z}$ and $\overline{U} \subseteq
  \pi^{-1}\bigl(\overline{\pi(U)}\bigr)$, we obtain
  \begin{equation} \label{eqPiU}%
    X = \overline{\pi\bigl(\overline{\mathcal{E}_Z}\bigr)} =
    \overline{\pi\bigl(\overline{U}\bigr)} \subseteq \overline{\pi(U)}
    \subseteq X,
  \end{equation}
  where the denseness of $\pi(\mathcal{E}_Z) = G \cdot Z$ was used for
  the first equality. Since $U \subseteq X \times K^n$ is a
  constructible subset it follows by theorems of Chevalley (see
  \mycite[Exercises~10.7 and~10.9, solutions given]{Kemper.Comalg})
  that $\pi(U)$ is also constructible and therefore contains a subset
  $V$ that is open and dense in $\overline{\pi(U)}$. So it follows
  from~\eqref{eqPiU} that $V$ is open and dense in $X$. There exists
  a nonempty open subset $V' \subseteq X$ such that $\dim(G \cdot x) =
  d$ for $x \in V'$. Since $V \cap V' \ne \emptyset$, we can take $x
  \in V \cap V'$. Then $x \in \pi(U)$, so there exists $i \in \{1
  \upto r\}$ and $v \in K^n$ such that $(x,v) \in
  U_i$. By~\eqref{eqDimGv} it follows that
  \[
  \dim\bigl(\overline{\mathcal{E}_i}\bigr) = \dim(G \cdot x) +
  \dim(Z_i) = d + \dim(Z_i) = d + \dim(Z),
  \]
  where the equidimensionality of $Z$ was used. Since $\dim\left(R[y_1
    \upto y_n]/E\right) = \dim\bigl(\overline{\mathcal{E}_Z}\bigr)$ is
  the maximal dimension of an $\overline{\mathcal{E}_i}$, the proof is
  complete.
\end{proof}

It may be interesting to note that the dimension formula in
\lref{lDimDerksen} fails for the extended Derksen ideal $E$
considered in \exref{exGm}. So the hypotheses from the lemma are not
unnecessarily restrictive.

We finish the section with the following lemma, which transports the
above result to Derksen ideals formed in a polynomial ring over the
function field. The lemma will be used in the next section.

\begin{lemma} \label{lDimL}%
  Let $G$ be a linear algebraic group, $X$ an irreducible $G$-variety,
  and $a_1 \upto a_n$ generators of $R := K[X]$. With $L := K(X)$, let
  $D_{a_1 \upto a_n}$ be the Derksen ideal formed in $L[y_1 \upto
  y_n]$. Then $L[y_1 \upto y_n]/D_{a_1 \upto a_n}$ is equidimensional
  of dimension equal to the maximal dimension~$d$ of a $G$-orbit in
  $X$.
\end{lemma}

\begin{proof}
  We have
  \[
  D_{a_1 \upto a_n} = \bigcap_{i=1}^m D_i
  \]
  with $D_i := \bigcap_{\sigma \in G^\circ} \bigl(y_1 - \sigma_i
  \sigma \cdot a_i\bigr)$, where the $\sigma_i$ are left coset
  representatives of $G^\circ$. Since the isomorphism given by
  $\sigma_i$ sends $D_1$ to $D_i$ (with the assumption $\sigma_1 \in
  G^\circ$), we may assume for the rest of the proof that $G =
  G^\circ$. Then $G \times X$ is irreducible, so $K[G] \otimes R$ is
  an integral domain, and the same follows for $K[G] \otimes L$. We
  have $g_1 \upto g_n \in K[G] \otimes R$ (tensor products are always
  over $K$) such that $\sigma^{-1} \cdot a_i = g_i(\sigma)$ for
  $\sigma \in G$. Consider the ideal $Q := (y_1 - g_1 \upto y_n - g_n)
  \subseteq K[G] \otimes L[y_1 \upto y_n]$, where the $g_i$ are as
  in~\eqref{eqActionA}. The isomorphism $(K[G] \otimes L[y_1 \upto
  y_n])/Q \cong K[G] \otimes L$ shows that $Q$ is a prime ideal. By
  \tref{tComputeDE}\eqref{tComputeDEA} below, we have
  \begin{multline*}
    \dim\left(L[y_1 \upto y_n]/D_{a_1 \upto a_n}\right) = \\
    \dim\left(L[y_1 \upto y_n]/(L[y_1 \upto y_n] \cap Q\right) = \\
    \trdeg_L\left(L[y_1 \upto y_n]/L[y_1 \upto y_n] \cap Q\right) = \\
    \trdeg_K\left(L[y_1 \upto y_n]/L[y_1 \upto y_n] \cap Q\right) -
    \trdeg_K(L).
  \end{multline*}
  But $R[y_1 \upto y_n]/R[y_1 \upto y_n] \cap Q$ is embedded into
  $L[y_1 \upto y_n]/L[y_1 \upto y_n] \cap Q$, and both rings share the
  same field of fractions and therefore the same transcendence degree
  over $K$. Therefore it suffices to show that $R[y_1 \upto y_n]/R[y_1
  \upto y_n] \cap Q$ has dimension $d + \trdeg_K(L) = d + \dim(X)$. We
  have $R[y_1 \upto y_n] \cap Q = R[y_1 \upto y_n] \cap D_{a_1 \upto
    a_n}^{(L)}$, where we put the superscript $(L)$ at the Derksen
  ideal in order to keep in mind that it is formed in $L[y_1 \upto
  y_n]$. We claim that
  \begin{equation} \label{eqDRL}%
    R[y_1 \upto y_n] \cap D_{a_1 \upto a_n}^{(L)} = D_{a_1 \upto
      a_n}^{(R)}.
  \end{equation}
  It is clear that $D_{a_1 \upto a_n}^{(R)} \subseteq R[y_1 \upto y_n]
  \cap D_{a_1 \upto a_n}$. Conversely, let $f \in R[y_1 \upto y_n]
  \cap D_{a_1 \upto a_n}^{(L)}$. Then for every $\sigma \in G$ there
  exists $b \in R$ nonzero such that $b f \in (y_1 - \sigma \cdot a_1
  \upto y_n - \sigma \cdot a_n) \subseteq R[y_1 \upto y_n]$. So $b
  \cdot f(\sigma \cdot a_1 \upto \sigma \cdot a_n) = 0$, which implies
  $f \in (y_1 - \sigma \cdot a_1 \upto y_n - \sigma \cdot a_n)
  \subseteq R[y_1 \upto y_n]$. We conclude that $f \in D_{a_1 \upto
    a_n}^{(R)}$. Now that~\eqref{eqDRL} is established, it remains to
  show that $\dim\bigl(R[y_1 \upto y_n]/D_{a_1 \upto a_n}^{(R)}\bigr)
  = d + \dim(X)$. But that is a special case of \lref{lDimDerksen}.
\end{proof}

\section{Extended Derksen ideals: computational
  aspects} \label{sDerksenC}%

How can (extended) Derksen ideals be calculated? For the ``classical''
Derksen ideal, an algorithm can be found in
\mycite[Section~4.1]{Derksen:Kemper}. The core is the computation of
an elimination ideal. As we will see, the same happens in a more
general situation, where we assume that a linear algebraic group $G$
over an algebraically closed field $K$ acts on a $K$-algebra $S$ by
automorphisms. Let $R := K[a_1 \upto a_n] \subseteq S$ be a finitely
generated subalgebra and assume that there exist $g_1 \upto g_n \in
K[G] \otimes R$ (tensor products are always over $K$) such that
\begin{equation} \label{eqActionA}%
  \sigma^{-1} \cdot a_i = g_i(\sigma)
\end{equation}
for $\sigma \in G$. This is a natural assumption. In fact, if $X$ is a
$G$-variety and $R = K[X]$, then the morphism $G \times X \to X$
defining the action induces a homomorphism $\map{\psi}{R}{K[G] \otimes
  R}$. Since $G$ acts on $R$ by $\sigma \cdot a = a \circ \sigma^{-1}$
for $a \in R$ and $\sigma \in G$, it follows by an easy calculation
that $\sigma^{-1} \cdot a = g(\sigma)$ with $g = \psi(a)$. In this
case we may take $S = R$ or $S = K(X) := \Quot(R)$ (if $X$ is
irreducible).

\begin{theorem} \label{tComputeDE}
  Assume the above notation and hypotheses.
  \begin{enumerate}
  \item \label{tComputeDEA} Let
    \[
    \widehat{D} := \bigl(y_1 - g_1 \upto y_n - g_n\bigr) \subseteq
    K[G] \otimes S[y_1 \upto y_n].
    \]
    Then the Derksen ideal is the elimination ideal
    \[
    D_{a_1 \upto a_n} = S[y_1 \upto y_n] \cap \widehat{D}.
    \]
  \item \label{tComputeDEB} Let $f_1 \upto f_s \in K[y_1
    \upto y_n]$ be polynomials and set
    \[
    \widehat{E} := \widehat{D} + \bigl(f_1 \upto f_s\bigr) \subseteq
    K[G] \otimes S[y_1 \upto y_n].
    \]
    If
    \begin{equation} \label{eqCondI}%
      R \cap \bigl(f_1(g_1 \upto g_n) \upto f_s(g_1 \upto g_n)\bigr)
      \subseteq \sqrt{\{0\}}
    \end{equation}
    (where the ideal $\bigl(f_1(g_1 \upto g_n) \upto f_s(g_1 \upto
    g_n)\bigr)$ is formed in $K[G] \otimes S$), then the elimination
    ideal
    \[
    E := S[y_1 \upto y_n] \cap \widehat{E}
    \]
    is a tamely extended Derksen ideal with respect to $a_1 \upto
    a_n$.
  \end{enumerate}
\end{theorem}

\begin{proof}
  \begin{enumerate}
  \item[\eqref{tComputeDEA}] If $f \in D_{a_1 \upto a_n}$, then for
    every $\sigma \in G$ we have $f\bigl(g_1(\sigma) \upto
    g_n(\sigma)\bigr) = 0$, so $f \in \widehat{D}$. Conversely, if $f
    \in S[y_1 \upto y_n] \cap \widehat{D}$, then $f = \sum_{i=1}^n h_i
    (y_i - g_i)$ with $h_i \in K[G] \otimes S[y_1 \upto y_n]$, so for
    $\sigma \in G$ we obtain $f = \sum_{i=1}^n h_i(\sigma) (y_i -
    \sigma^{-1} \cdot a_i)$, and then $f(\sigma^{-1} \cdot a_1 \upto
    \sigma^{-1} \cdot a_n) = 0$. This implies $f \in D_{a_1 \upto
      a_n}$.
  \item[\eqref{tComputeDEB}] Let $G$ act on $K[G]$ by defining $\tau
    \cdot f$ as the function $G \to K$, $\sigma \mapsto f(\sigma
    \tau)$, where $\tau \in G$ and $f \in K[G]$. It is easy to check
    that the $g_i$ are $G$-invariant under the action on $K[G] \otimes
    R$. The $f_i$ are also invariant. It follows that $\widehat{E}$ is
    $G$-stable, hence the same is true for $E$. It follows
    from~\eqref{tComputeDEA} that $D_{a_1 \upto a_n} \subseteq
    E$. Consider the ideal
    \[
    I := \bigl\{h(a_1 \upto a_n) \mid h \in K[y_1 \upto y_n] \cap
    E\bigr\} \subseteq R.
    \]
    from \dref{dDerksen}\eqref{dDerksenC} and take a nonzero element
    $d \in \bigcap_{\sigma \in G} \sigma \cdot I =:J$. We need to show
    that~$d$ is nilpotent. Since the $G$-action on $R$ is locally
    finite (see \mycite[Chapter~1, \S~1, Lemma]{MFK}) and $J$ is
    $G$-stable, it contains a $G$-module $V$ with $d \in V$. Choose a
    basis $d = d_1,d_2 \upto d_m$ of $V$. There exists a matrix
    $(c_{j,i}) \in K[G]^{m \times m}$ such that $\sigma \cdot d_i =
    \sum_{j=1}^m c_{j,i}(\sigma) d_j$ for $1 \le i \le m$ and $\sigma
    \in G$. Since $d_j \in J \subseteq I$ we can write $d_j = h_j(a_1
    \upto a_n)$ with $h_j \in K[y_1 \upto y_n] \cap E$. We obtain
    \[
    d_i = \sum_{j=1}^m c_{j,i}(\sigma) \bigr(\sigma^{-1} \cdot
    d_j\bigr) = \sum_{j=1}^m c_{j,i}(\sigma) h_j\bigl(g_1(\sigma)
    \upto g_n(\sigma)\bigr),
    \]
    which implies
    \[
    d_i = \sum_{j=1}^m c_{j,i} h_j(g_1 \upto g_n) \in K[G] \otimes R.
    \]
    From $h_j \in E$ we obtain $h_j(g_1 \upto g_n) \in \widehat{E}$,
    so $d_i \in R \cap \widehat{E}$. In particular, this holds for $d
    = d_1$. So if we can show that $R \cap \widehat{E}$ is nilpotent,
    we are done. We have
    \[
    \widehat{E} = \widehat{D} + \bigl(f_1(g_1 \upto g_n) \upto f_s(g_1
    \upto g_n)\bigr),
    \]
    so
    \[
    \bigl(K[G] \otimes S\bigr) \cap \widehat{E} = \bigl(f_1(g_1 \upto
    g_n) \upto f_s(g_1 \upto g_n)\bigr).
    \]
    This implies $R \cap \widehat{E} = R \cap \bigl(f_1(g_1 \upto g_n)
    \upto f_s(g_1 \upto g_n)\bigr)$, which is nilpotent by
    hypothesis. \qed
  \end{enumerate}
  \renewcommand{\qed}{}
\end{proof}

\begin{rem} \label{rComputeDE}%
  We wish to give a geometric interpretation to the
  hypothesis~\eqref{eqCondI}. Consider the ideal
  \[
  \widehat{E}_R := \bigl(y_1 - g_1 \upto y_n - g_n,f_1 \upto f_s\bigr)
  \subseteq K[G] \otimes R[y_1 \upto y_n].
  \]
  Assume that $R = K[X]$ and let $Z \subseteq X$ be the closed subset
  given by $f_i(a_1 \upto a_n) = 0$ for $i = 1 \upto s$. Then
  $\widehat{E}_R$ defines the closed subset
  \[
  \widehat{\mathcal{E}} := \bigl\{(\sigma^{-1},\sigma \cdot z,\phi(z))
  \mid \sigma \in G, \ z \in Z\bigr\} \subseteq G \times X \times K^n.
  \]
  So $R \cap \widehat{E}_R$ defines the closure of the projection
  $\pi_X(\widehat{\mathcal{E}}) = G \cdot Z$ of
  $\widehat{\mathcal{E}}$ to $X$. It follows that $R \cap
  \widehat{E}_R = \{0\}$ if and only if $G \cdot Z$ is dense in
  $X$. We have
  \[
  R \cap \widehat{E}_R = R \cap \bigl(f_1(g_1 \upto g_n) \upto f_s(g_1
  \upto g_n)\bigr)_{K[G] \otimes R},
  \]
  where the index at the bracket signifies the ring in which the ideal
  is formed. If $X$ is irreducible and $L = K(X)$, then
  \[
  R \cap \bigl(f_1(g_1 \upto g_n) \upto f_s(g_1 \upto g_n)\bigr)_{K[G]
    \otimes R} = \{0\}
  \]
  is equivalent to~\eqref{eqCondI}. So if $S = K[X]$ or $S = K(X)$
  (if $X$ is irreducible), then~\eqref{eqCondI} is equivalent to the
  condition that $G \cdot Z$ is dense in $X$.
\end{rem}

The following result tells us how far we can get with tamely extended
Derksen ideals. The upshot is that they can be chosen in such a way
that the number of indeterminates involved in the computation of the
elimination ideal $E$ of $\widehat{E}$ (see \tref{tComputeDE}) is
effectively reduced by~$d$, the maximal dimension of a $G$-orbit. The
theorem generalizes Theorem~3.3 from \mycite{HK:06}, which only
applies to affine $n$-space. It also uses a selection of the $y_i$
rather than linear combinations of them.

\begin{theorem} \label{tCrossSection}%
  Let $G$ be a linear algebraic group, $X$ an irreducible $G$-variety,
  and $a_1 \upto a_n$ generators of $K[X]$. With $L := K(X)$, set
  \[
  \widehat{D} := \bigl(y_1 - g_1 \upto y_n - g_n\bigr) \subseteq K[G]
  \otimes L[y_1 \upto y_n]
  \]
  as in \tref{tComputeDE}. Let~$d$ be the maximal dimension of a
  $G$-orbit in $X$. Then there exist indices $1 \le i_1 < \cdots < i_d
  \le n$ such that the classes of the $y_{i_j}$ in $\bigl(K[G] \otimes
  L[y_1 \upto y_n]\bigr)/\widehat{D}$ are algebraically independent
  over $L$. Moreover, there exists an open, dense subset $U \subseteq
  K^d$ such that for $(\beta_1 \upto \beta_d) \in U$ the polynomials
  \[
  f_j := y_{i_j} - \beta_j \in K[y_1 \upto y_n] \quad (j = 1 \upto d)
  \]
  satisfy the condition~\eqref{eqCondI} from
  \tref{tComputeDE}. Moreover, with
  \[
  \widehat{E} := \widehat{D} + (f_1 \upto f_d) \quad \text{and} \quad
  E := L[y_1 \upto y_n] \cap \widehat{E},
  \]
  we have
  \[
  \dim\left(\bigl(K[G] \otimes L[y_1 \upto
    y_n]\bigr)/\widehat{E}\right) = \dim(G) - d
  \]
  and $\dim\left(L[y_1 \upto y_n]/E\right) = 0$.
\end{theorem}

\begin{proof}
  From the isomorphism
  \[
  \bigl(K[G] \otimes L[y_1 \upto y_n]\bigr)/\widehat{D} \cong K[G]
  \otimes L
  \]
  we conclude that $\bigl(K[G] \otimes L[y_1 \upto
  y_n]\bigr)/\widehat{D}$ is equidimensional of the same dimension as
  $G$. Moreover, from \lref{lDimL} we know that $L[y_1 \upto y_n]/D$
  is equidimensional of dimension~$d$, where $D := L[y_1 \upto y_n]
  \cap \widehat{D}$. (By \tref{tComputeDE}\eqref{tComputeDEA}, $D$
  is the Derksen ideal with respect to the $a_i$ formed in $L[y_1
  \upto y_n]$.) This implies that there exist $1 \le i_1 < \cdots <
  i_d \le n$ such that the $y_{i_j} + D \in L[y_1 \upto y_n]/D$ are
  algebraically independent (see, for example, \mycite[Theorem~5.9 and
  Proposition~5.10]{Kemper.Comalg}). With $A := L[y_{i_1} \upto
  y_{i_d}]$, we get injective maps
  \[
  A \overset{\phi}{\longrightarrow} L[y_1 \upto y_n]/D
  \overset{\psi}{\longrightarrow} \bigl(K[G] \otimes L[y_1 \upto
  y_n]\bigr)/\widehat{D}
  \]
  By a standard result on the dimension of fibers (see
  \mycite[Theorem~10.5]{Kemper.Comalg}), there exists a nonzero $a \in
  A$ such that for every maximal ideal $\mathfrak{m} \subset A$ with
  $a \notin \mathfrak{m}$ the fibers of~$\mathfrak{m}$ in
  $\Spec\bigl(L[y_1 \upto y_n]/D\bigr)$ and in $\Spec\bigl(\bigl(K[G]
  \otimes L[y_1 \upto y_n]\bigr)/\widehat{D}\bigr)$ are nonempty and
  have dimensions $\dim\bigl(L[y_1 \upto y_n]/D\bigr) - \dim(A) = 0$
  and $\dim\bigl(\bigl(K[G] \otimes L[y_1 \upto
  y_n]\bigr)/\widehat{D}\bigr) - \dim(A) = \dim(G) - d$,
  respectively. (Here the equidimensionality is used.) This means that
  for the ideals $I := \bigl(\phi(\mathfrak{m})\bigr) \subseteq L[y_1
  \upto y_n]$ and $J := \bigl(\psi(\phi(\mathfrak{m}))\bigr) \subseteq
  K[G] \otimes L[y_1 \upto y_n]$ one has
  \[
  \dim\left(K[y_1 \upto y_n]/I\right) = 0, \ \dim\left(\bigl(K[G]
    \otimes L[y_1 \upto y_n]\bigr)/J\right) = \dim(G) - d.
  \]
  In particular, this holds for $\mathfrak{m} = (y_{i_1} - \beta_1
  \upto y_{i_d} - \beta_d) \subset A$ with $\beta_1 \upto \beta_d \in
  K$ such that $a(\beta_1 \upto \beta_d) \ne 0$. With $f_j := y_{i_j}
  - \beta_j$, we obtain that $D + (f_1 \upto f_d)$ and $\widehat{E}$
  (as defined in the statement of the theorem) are proper ideals, with
  dimensions as above. Since
  \[
  \big(K[G] \otimes L\bigr) \cap \widehat{E} = \bigl(g_{i_1} - \beta_1
  \upto g_{i_d} - \beta_d\bigr),
  \]
  it follows that the~$f_j$ satisfy the condition~\eqref{eqCondI}
  from \tref{tComputeDE}. Since $D + (f_1 \upto f_d) \subseteq E$, we
  also get $\dim\bigl(L[y_1 \upto y_n]/E\bigr) = 0$.
\end{proof}

\begin{rem} \label{rInvariantization}%
  The last equality in \tref{tCrossSection} is significant since it
  leads to a variant of the invariantization map that does not depend
  on the choice of the monomial ordering used for the computation of a
  Gr\"obner basis of $E$. Indeed, assume that $E \subseteq L[y_1 \upto
  y_n]$ is an extended Derksen ideal such that $N := L[y_1 \upto
  y_n]/E$ has Krull dimension zero. Then its dimension $e :=
  \dim_L(N)$ as an $L$-vector space is finite. We assume that~$e$ is
  not a multiple of $\ch(K)$. For $b = f(a_1 \upto a_n) \in K[a_1
  \upto a_n]$ with $f \in K[y_1 \upto y_n]$, define $\phi_E(b)$ as the
  trace of the endomorphism of $N$ given by multiplication by $e^{-1}
  f$. Then \tref{tInvariantization} holds with~$\phi_\mathcal{G}$
  replaced by~$\phi_E$. In fact, the proof of the theorem carries over
  to this case.

  Although~$\phi_E$ is independent of the choice of a monomial
  ordering, it does depend on the choice of $E$. This can be seen by
  reconsidering
  \exref{exInvariantization}\eqref{exInvariantization1} and using
  $E' = D_{x_1,x_2} + (y_2 - 1)$ as an alternative extended Derksen
  ideal. We have $\phi_E = \phi_\mathcal{G}$ if and only if $e = 1$,
  and in this case $\phi_E$ is a homomorphism of $R^G$-algebras.
\end{rem}

We are now ready to let our results (in particular,
Theorems~\ref{tInvariantField}, \ref{tLocalize}, \ref{tComputeDE},
and~\ref{tCrossSection}) flow into an algorithm for the computation of
invariant fields and localizations of invariant rings. The result of
the algorithm can be fed into Semi-algorithm~\ref{aSemi}. We assume
the standard situation of invariant theory with $K$ an algebraically
closed field.

\begin{alg}[Computation of a localization of an invariant
  ring] \label{aLocalize} \mbox{}%
  \begin{description}
  \item[\bf Input:] A linear algebraic group $G$ given as a subset of
    $K^m$ by a radical ideal $I_G \subseteq K[z_1 \upto z_m]$, and an
    irreducible $G$-variety $X$ given by a prime ideal $I_X \subseteq
    K[x_1 \upto x_n]$, with the action given by
    \[
    \sigma \cdot v = \bigl(g_1(v,\sigma) \upto g_n(v,\sigma)\bigr)
    \]
    for $v \in X$ and $\sigma \in G$, where $g_i \in K[x_1 \upto
    x_n,z_1 \upto z_m]$.
  \item[\bf Output:] Generators of the invariant field $K(X)^G$ and
    invariants $a,b_1 \upto b_k \in K[X]^G$ such that
    \[
    K[X]^G_a = K[a^{-1},a,b_1 \upto b_k].
    \]
    The latter is only possible if $K(X)^G = \Quot\left(K[X]^G\right)$
    (see \tref{tItalian} for conditions which guarantee this), or if
    step~\ref{aLocalize4} is used in such a way as to ensure the
    existence of $a \in K[X]^G$ as in \tref{tLocalize}.
  \end{description}
  \begin{enumerate}
    \renewcommand{\theenumi}{\arabic{enumi}}
  \item \label{aLocalize1} The first step is optional but
    recommended. It implements \tref{tCrossSection}. Set~$d$ equal to
    (or less than) the maximal dimension of a $G$-orbit in $X$. Choose
    an injective map $\map{\eta}{\NN_0}{K}$.

    For all $s = 0,1,2, \ldots$, for all $(a_1 \upto a_d) \in \NN_0^d$
    with $\sum_{i=1}^d a_i = s$, and for all $1 \le i_1 < i_2 < \cdots
    < i_d \le n$, let $J \subseteq K[x_1 \upto x_n,z_1 \upto z_m]$ be
    the ideal generated by $I_X$, $I_G$ and the $g_{i_j} - \eta(a_j)$
    ($j = 1 \upto d$). Check whether
    \[
    K[x_1 \upto x_n] \cap J \subseteq I_X.
    \]
    When this condition is satisfied, remember the $i_1 \upto i_d$,
    set $\beta_j := \eta(a_j)$, and proceed to the next step.
  \item With $y_1 \upto y_n$ additional indeterminates, form the ideal
    $\widehat{E} \subseteq K[x_1 \upto x_n,y_1 \upto y_n,\linebreak z_1 \upto
    z_m]$ generated by $I_X$, $I_G$, $y_i - g_i$ ($i = 1 \upto n$),
    and (if step~\ref{aLocalize1} was not omitted) $y_{i_j} -
    \beta_j$ ($j = 1 \upto d$).

    Choose a monomial ordering on $K[y_1 \upto y_n,z_1 \upto z_m]$
    such that every $z_i$ is bigger than every power of a $y_j$, then
    choose an arbitrary monomial ordering on $K[x_1 \upto x_n]$, and
    let $>$ be the block ordering on $K[x_1 \upto x_n,y_1 \upto
    y_n,z_1 \upto z_m]$ formed from these two orderings, with
    precedence on the $y$- and $z$-variables. For example, a
    lexicographic ordering with $z_i > y_j > x_l$ for all $i,j,l$ is
    possible.
  \item \label{aLocalize3} Compute a Gr\"obner basis
    $\widehat{\mathcal{G}}$ of $\widehat{E}$ with respect to $>$. Then
    set
    \[
    \mathcal{G} := K[x_1 \upto x_n,y_1 \upto y_n] \cap
    \widehat{\mathcal{G}},
    \]
    $\mathcal{G}_x := K[x_1 \upto x_n] \cap \mathcal{G}$, and
    $\mathcal{G}_y := \mathcal{G} \setminus \mathcal{G}_x$. Viewed as
    polynomials in $K(X)[y_1 \upto y_n]$, the elements of
    $\mathcal{G}_y$ form a Gr\"obner basis of the tamely extended
    Derksen ideal given by $\widehat{E}$. Moreover, $\mathcal{G}_x$ is
    a Gr\"obner basis of $I_X$.
  \item \label{aLocalize4} This step is optional and should only be
    tried if $K(X)^G \ne \Quot\left(K[X]^G\right)$. Choose an ideal
    $I_Z \subseteq K[y_1 \upto y_n]$ such that the ideal $J \subseteq
    K[x_1 \upto x_n,y_1 \upto y_n]$ generated by $\mathcal{G}$ and
    $I_Z$ satisfies
    \[
    K[x_1 \upto x_n] \cap J \subseteq I_X
    \]
    The goal of this step is to obtain a nontamely extended Derksen
    ideal such that $a \in K[X]^G$ as in \tref{tLocalize}
    exists. After choosing $I_Z$, replace $\mathcal{G}$ by a Gr\"obner
    basis of $J$ and set $\mathcal{G}_x$ and $\mathcal{G}_y$ as in
    step~\ref{aLocalize3}.
  \item This step turns $\mathcal{G}_y$, viewed as a subset of
    $K(X)[y_1 \upto y_n]$, into a reduced Gr\"obner basis. For all $f
    \in \mathcal{G}_y$ perform step~\ref{aLocalize6}.
  \item \label{aLocalize6} As long as there exists $g \in
    \mathcal{G}_y \setminus \{f\}$ such that $\LM(g)$ divides a term
    $c \cdot m$ of~$f$ (where~$f$ and~$g$ are viewed as polynomials in
    the~$y_i$ with coefficients in $K[x_1 \upto x_n]$), take the
    maximal such monomial~$m$ and replace~$f$ by
    \[
    \NF_{\mathcal{G}_x}\bigl(\LC(g) f - c g\bigr).
    \]
    If this is zero, delete~$f$ from $\mathcal{G}_y$.
  \item \label{aLocalize7} Let $\frac{f_j}{h_j}$ ($j = 1 \upto k$,
    $f_j,h_j \in K[x_1 \upto x_n]$) be the coefficients appearing in
    the polynomials $\LC(f)^{-1} f$ with $f \in \mathcal{G}_y$. If
    step~\ref{aLocalize4} was omitted (the standard case), then
    \[
    K(X)^G = K\left(\frac{f_1 + I_X}{h_1 + I_X} \upto \frac{f_k +
        I_X}{h_k + I_X}\right).
    \]
  \item This step searches invariants $a,b_1 \upto b_k \in K[X]^G$
    such that $\frac{f_j + I_X}{h_j + I_X} = \frac{b_j}{a}$. This only
    terminates if such invariants exist, which is guaranteed if
    $K(X)^G = \Quot\left(K[X]^G\right)$. For $r = 0,1,2, \ldots$
    perform steps~\ref{aLocalize9} and~\ref{aLocalize10}.
  \item \label{aLocalize9} Let $m_1 \upto m_l \in K[x_1 \upto x_n]$
    be all monomials of degree $\le r$ that are in normal form with
    respect to $\mathcal{G}_x$. Consider the system of linear
    equations for $\alpha_i$ and $\beta_{i,j} \in K$ ($i = 1 \upto l$,
    $j = 1 \upto k$) given by
    \[
    \sum_{i=1}^l \alpha_i \NF_{\mathcal{G}_x}(f_j m_i) = \sum_{i=1}^l
    \beta_{i,j} \NF_{\mathcal{G}_x}(h_j m_i) \quad (j = 1 \upto k)
    \]
    and
    \[
    \sum_{i=1}^l \alpha_i
    \NF_{\mathcal{G}_G}\left(\NF_{\mathcal{G}_x}\bigl(m_i(g_1 \upto
      g_n) - m_i\bigr)\right) = 0,
    \]
    where $\mathcal{G}_G \subseteq K[z_1 \upto z_m]$ is a Gr\"obner
    basis of $I_G$.
  \item \label{aLocalize10} If the system has a nonzero solution,
    then with $a := \sum_{i=1}^l \alpha_i m_i + I_X$ and $b_j :=
    \sum_{i=1}^l \beta_{i,j} m_i + I_X \in K[X]$ we have
    \[
    K[X]^G_a = K[a^{-1},a,b_1 \upto b_k].
    \]
  \end{enumerate}
\end{alg}

Instead of directly computing a Gr\"obner basis over the function
field $L = K(X)$, the algorithm computes in an appropriate polynomial
ring over $K$. This has two advantages: First, computer algebra
systems do not support Gr\"obner basis computations over fields as
complicated as function fields of irreducible varieties. And second,
even if $X = K^n$ (and so $K(x_1 \upto x_n)$ is supported as a ground
field for Gr\"obner basis computation), experience shows that it is
better to perform the computations in a polynomial ring. By
remembering the polynomials $\LC(f)^{-1} f$ with $f \in \mathcal{G}_y$
(which, viewed as polynomials in $K(X)[y_1 \upto y_n]$, form a reduced
Gr\"obner basis of the extended Derksen ideal), one can also get the
invariantization map from \tref{tInvariantization} out of
\aref{aLocalize}. \\

The special case of the additive group $\Ga$ is particularly easy to
deal with. The following algorithm computes a localization of the
invariant ring of $\Ga$ under mild hypotheses. The first algorithm for
computing invariants of the additive groups was given by
\mycite{essen}. His algorithm is essentially \aref{aGa}. See in
\mycite{Freudenburg:2006} for a much more comprehensive treatment.

\begin{alg}[A localization of the invariant ring of a
  $\Ga$-action] \label{aGa} \mbox{}%
  \begin{description}
  \item[\bf Input:] An irreducible affine variety $X$ given by a prime
    ideal $I_X \subseteq K[x_1 \upto x_n]$, with an action of the
    additive group $G = \Ga$ given by
    \[
    t \cdot v = \bigl(g_1(v,t) \upto g_n(v,t)\bigr)
    \]
    for $v \in X$ and $t \in \Ga$, where $g_i \in K[x_1 \upto
    x_n,z]$. With the $g_i$ chosen in such a way that no coefficient
    of a $g_i$ (as a polynomial in~$z$) lies in $I_X$, at least one
    $g_i$ is assumed to have a degree not divisible by $\ch(K)$. If
    $\ch(K) = 0$, this hypothesis just means that the $\Ga$-action is
    nontrivial.
  \item[\bf Output:] Invariants $a,b_1 \upto b_n \in K[X]^\Ga$ such
    that
    \[
    K[X]^\Ga_a = K[a^{-1},a,b_1 \upto b_n],
    \]
    and an invariantization map
    \[
    \map{\phi}{K[X]_a}{K[X]^\Ga_a},
    \]
    which is a homomorphism of $K[X]^\Ga_a$-algebras.
  \end{description}
  \begin{enumerate}
    \renewcommand{\theenumi}{\arabic{enumi}}
  \item For $i = 1 \upto n$, let $d_i$ be the degree of $g_i$ as a
    polynomial in~$z$. Choose an~$i$ such that~$d_i$ is not divisible
    by $\ch(K)$. Write
    \[
    g_i = \sum_{j=0}^{d_i} g_{i,j} \cdot z^{d_i-j}
    \]
    with $g_{i,j} \in K[x_1 \upto x_n]$.
  \item For $j = 1 \upto n$, let $h_j \in K(x_1 \upto x_n)$ be the
    result of substituting
    \[
    z = \frac{- g_{i,1}}{d_i \cdot g_{i,0}}
    \]
    in $g_j$.
  \item Set $a := g_{i,0} + I_X$ and $b_j := g_{i,0}^{d_j} h_j +
    I_X$. Define $\map{\phi}{K[X]_a}{K[X]^\Ga_a}$ as a homomorphism of
    $K$-algebras by
    \[
    \phi(a^{-1}) = a^{-1} \quad \text{and} \quad \phi(x_j + I_X) =
    a^{-d_j} b_j.
    \]
  \end{enumerate}
\end{alg}

Of course, \aref{aLocalize} is applicable to all $\Ga$-actions
without the assumption made in \aref{aGa}. But without this assumption
the computation will be harder and the result will be less easy to
describe. Another algorithm for computing $\Ga$-invariants in all
characteristics was given by
\mycite[Section~3.1.2]{Derksen.Kemper06}. This reference also contains
an example of a nontrivial $\Ga$-action for which the assumption of
\aref{aGa} is not satisfied.

\begin{proof}[Proof of correctness of \aref{aGa}]
  If $f \in K[x_1 \upto x_n]$ we write $\overline{f} = f + I_X$ for
  the image of~$f$ in $K[X]$. For $s,t \in \Ga$ we have
  \[
  (-s - t) \cdot \overline{x}_i = (-t) \cdot \bigl(-s \cdot
  \overline{x}_i\bigr) = (-t) \cdot \Bigl(\sum_{j=0}^{d_i}
  \overline{g}_{i,j} s^{d_i - j}\Bigr) = \sum_{j=0}^{d_i} \bigl(-t
  \cdot \overline{g}_{i,j}\bigr) s^{d_i - j},
  \]
  and on the other hand
  \[
  (-s - t) \cdot \overline{x}_i = \sum_{j=0}^{d_i} \overline{g}_{i,j}
  (s + t)^{d_i - j} = \sum_{k=0}^{d_i} \sum_{j=0}^{d_i-k}
  \overline{g}_{i,j} \binom{d_i-j}{k} t^{d_i-j-k} s^k.
  \]
  Since these formulas hold for all~$s$, comparison of the
  coefficients of $s^{d_i}$ and $s^{d_i-1}$ yields
  \begin{equation} \label{eqLinear}%
    (-t) \cdot \overline{g}_{i,0} =
    \overline{g}_{i,0} \quad \text{and} \quad (-t) \cdot
    \overline{g}_{i,1} = \overline{g}_{i,1} + d_i \overline{g}_{i,0}
    t.
  \end{equation}
  So $a = \overline{g}_{i.0}$ is a (nonzero) invariant. In the
  function field $L := K(X)$ we consider the elements $a_0 :=
  \overline{g}_{i,1}$ and $a_i := \overline{x}_i$ ($i = 1 \upto
  n$). If we set $g_0 := g_{i,1} + d_i g_{i,0} z \in K[x_1 \upto
  x_n,z]$, then for $t \in \Ga$ and $i \in \{0 \upto n\}$ we have
  $(-t) \cdot a_i = \overline{g}_i(t)$. Clearly $L \cap
  (\overline{g}_0) = \{0\}$, so \tref{tComputeDE}\eqref{tComputeDEB}
  tells us that with
  \[
  \widehat{E} = \Bigl(y_0,\overline{g}_0,y_1 - \overline{g}_1 \upto
  y_n - \overline{g}_n\Bigr) \subseteq K[\Ga] \otimes L[y_0 \upto
  y_n],
  \]
  the intersection $E := L[y_1 \upto y_n] \cap \widehat{E}$ is a
  tamely extended Derksen ideal. Using the map
  \[
  \mapl{\psi}{L[z]}{L}{f(z)}{f\Bigl(\frac{-\overline{g}_{i.1}}{d_i
      \overline{g}_{i,0}}\Bigr)},
  \]
  we can write $\widehat{E}$ as
  \[
  \widehat{E} = \Bigl(y_0,z - \psi(z),y_1 - \psi(\overline{g}_1) \upto
  y_n - \psi(\overline{g}_n)\Bigr).
  \]
  Now we see that the given generators form a reduced Gr\"obner basis,
  so $E$ has the reduced Gr\"obner basis
  \[
  \mathcal{G} = \bigl\{y_0,y_1 - \psi(\overline{g}_1) \upto y_n -
  \psi(\overline{g}_n)\bigr\}.
  \]
  Using the notation of the algorithm, we have $\psi(\overline{g}_j) =
  a^{-d_j} b_j$. So \tref{tLocalize} yields
  \[
  K[X]^\Ga_a = K[a^{-1},a,a^{-d_1} b_1 \upto a^{-d_n} b_n]
  \]
  and hence also $K[X]^\Ga_a = K[a^{-1},a,b_1 \upto b_n]$. Moreover,
  for $f \in K[x_1 \upto x_n]$ we have $\NF_{\mathcal{G}}\bigl(f(y_1
  \upto y_n)\bigr) = f\bigl(a^{-d_1} b_1 \upto a^{-d_n} b_n\bigr)$, so
  the map~$\phi$ from the algorithm is indeed the invariantization map
  from \tref{tInvariantization}. The theorem implies that~$\phi$ is
  constant on $K[X]^\Ga_a$, and~$\phi$ is a ring homomorphism by
  construction.
\end{proof}

Let us look at an example. The (in some sense) smallest example known
to date of a nonfinitely generated invariant ring was given by
\mycite{Daigle.Freudenburg99}. So it will be interesting to run
\aref{aGa} on this example.

\begin{ex} \label{exDF}%
  Daigle and Freudenburg's example is an action of the additive group
  $\Ga$ on the polynomial ring $R = \CC[x_1 \upto x_5]$ (and on its
  field of fractions $L = \CC(x_1 \upto x_5)$), which is best defined
  in terms of the locally nilpotent derivation
  \[
  \delta = x_1^3 \frac{\partial}{\partial x_2} + x_2
  \frac{\partial}{\partial x_3} + x_3 \frac{\partial}{\partial x_4} +
  x_1^2 \frac{\partial}{\partial x_5},
  \]
  so
  \[
  \CC[x_1 \upto x_5]^\Ga = \ker(\delta).
  \]
  Converting the action to make it compatible with our setting yields
  an action given by $(-t) \cdot x_i = g_i(x_1 \upto x_5,t)$ for $t
  \in \Ga = \CC$ with polynomials
  \begin{align*}
    & g_1 = x_1, \quad g_2 = x_2 + z x_1^3, \quad g_3 = x_3 + z x_2 +
    \frac{z^2}{2} x_1^3, \\
    & g_4 = x_4 + z x_3 + \frac{z^2}{2} x_2 + \frac{z^3}{6} x_1^3,
    \quad \text{and} \quad g_5 = x_5 + z x_1^2.
  \end{align*}
  As a polynomial in~$z$, $g_2$ has degree~$1$, so we need to
  substitute $z = - x_2/x_1^3$ in the $g_i$. This yields rational
  invariants
  \[
  h_1 = x_1, \ h_2 = 0, \ h_3 = \frac{2 x_1^3 x_3 - x_2^2}{2 x_1^3}, \
  h_4 = \frac{3 x_1^6 x_4 - 3 x_1^3 x_2 x_3 + x_2^3}{3 x_1^6}, \
  \text{and} \ h_5 = \frac{x_1 x_5 - x_2}{x_1}.
  \]
%
  We obtain
  \[
  \CC[x_1 \upto x_5]^\Ga_{x_1} = \CC[x_1^{-1},x_1,f_1,f_2,f_3]
  \]
  with
  \[
  f_1 = 2 x_1^3 x_3 - x_2^2, \quad f_2 = 3 x_1^6 x_4 - 3 x_1^3 x_2 x_3
  + x_2^3, \quad \text{and} \quad f_3 = x_1 x_5 - x_2.
  \]
  So the localized invariant ring is isomorphic to a localized
  polynomial ring, the simplest possible structure. In particular, the
  invariant field is $\CC(x_1 \upto x_5)^\Ga = \CC(x_1,f_1,f_2,f_3)$,
  so it is a purely transcendental field extension of $\CC$. It seems
  amazing that despite all this simplicity, the invariant ring itself
  is not finitely generated.

  The invariantization map~$\phi$ is given by
  \[
  \phi(x_1) = x_1, \ \phi(x_2) = 0, \ \phi(x_3) = \frac{f_1}{2 x_1^3},
  \ \phi(x_4) = \frac{f_2}{3 x_1^6}, \ \text{and} \ \phi(x_5) =
  \frac{f_3}{x_1}.
  \]
  This is not $\Ga$-equivariant. (In fact, there cannot exist a
  $\Ga$-equivariant projection \linebreak $\CC[x_1^{-1},x_1 \upto x_5] \to
  \CC[x_1 \upto x_5]^\Ga_{x_1}$ since such a map would produce a
  complement of the invariant ring, which would then contain nonzero
  invariants by the unipotency of $\Ga$.)

  According to \tref{tRewrite} we also get the following procedure
  for rewriting an invariant in terms of $x_1$ and the $f_i$. For $f
  \in \CC[x_1^{-1},x_1 \upto x_5]$ form
  \[
  \widetilde{f} = f\left(t_0,0,\frac{t_1}{2 t_0^3},\frac{t_2}{3
      t_0^6},\frac{t_3}{t_0}\right) \in \CC[t_0^{-1},t_0 \upto t_3],
  \]
  where the $t_i$ are indeterminates. Then
  $\widetilde{f}(x_1,f_1,f_2,f_3) = \phi(f)$, so~$f$ is an invariant
  if and only if $f = \widetilde{f}(x_1,f_1,f_2,f_3)$. Hence
  $\widetilde{f}$ represents an invariant~$f$ in terms of the
  generating invariants.
\end{ex}

\aref{aGa} and its proof of correctness have an interesting geometric
interpretation. Consider the subsets
\[
U := \bigl\{x \in X \mid g_{i,0}(x) \ne 0\bigr\} \quad \text{and}
\quad S := \bigl\{x \in U \mid g_{i,1}(x) = 0\bigr\}
\]
of $X$. We see from~\eqref{eqLinear} that $U$ is $\Ga$-stable (so it
is a $\Ga$-variety) and that every $\Ga$-orbit in $U$ meets $S$ at
precisely one point. In fact, the map
\[
\Ga \times S \to U, \ (t,x) \mapsto t \cdot x
\]
is an isomorphism with
\[
U \to \Ga \times S, \ x \mapsto \left(\frac{g_{i,1}(x)}{d_i
    g_{i,0}(x)},\frac{-g_{i,1}(x)}{d_i g_{i,0}(x)} \cdot x\right)
\]
as inverse map. With $\Ga$ acting on itself by left translation and
trivially on $S$, the map $\Ga \times S \to U$ is actually an
isomorphism of $\Ga$-varieties. The second projection $\Ga \times S
\to S$ is a geometric quotient (in the sense of
\mycite[Definition~0.6]{MFK}), so we obtain a commutative diagram%
\DIAGV{80}%
{\Ga \times S} \n{\Ear{\sim}} \n{U} \n{\ear} \n{X} \nn%
{\sar} \n{} \n{\sar} \n{} \n{\sar} \nn%
S \n{\Ear{\sim}} \n{U \quo \Ga} \n{\ear} \n{X \quo \Ga}%
\diag%
It follows that $U \to U \quo \Ga$ is also a geometric quotient. This
provides an example for a theorem of \mycite{Rosenlicht:63}. Using the
above diagram, it is not hard to see that the map $S \to X \quo \Ga$
is \'etale and that $\Ga \times S \cong X \times_{X \quo \Ga} S$. This
means that $S$ is a slice in the sense of \mycite{luna}. This is
probably the reason why \mycite{Freudenburg:2006} and other authors use
the term {\em local slice} for an element from $K[X]$ on which the
$\Ga$-action is given by a polynomial of degree~$1$ in~$t$, as
in~\eqref{eqLinear}.

The invariantization map $\map{\phi}{K[U]}{K[U]^{\Ga}}$ arises as
follows: The map $S \to \Ga \times S$, $x \mapsto (0,x)$ is a right
inverse of the quotient $\Ga \times S \to S$. Using the isomorphisms
in the above diagram, we obtain a right inverse $U \quo \Ga \to U$ of
the quotient $U \to U \quo \Ga$, from which~$\phi$ arises. In more
explicit terms, for $f \in K[U]$, the function $\map{\phi(f)}{U}{K}$
is defined by sending a point $x \in U$ to the evaluation of~$f$ at
$\frac{-g_{i,1}(x)}{d_i g_{i,0}(x)} \cdot x$, which is the unique
point where the orbit $\Ga \cdot x$ meets $S$.

\section{Linear actions} \label{sLinear}

In this section we consider a linear algebraic group $G$ with a
$G$-module $V$. In this case, Semi-algorithm~\ref{aSemi} for
extracting the invariant ring $K[V]^G$ from a localization $K[V]^G_a$
can be optimized, and the range of applicability of our algorithms can
be broadened. We start by giving a variant of the semi-algorithm for
``unlocalizing'', which is divided into two parts, \aref{aTest} and
Semi-algorithm~\ref{aUnlocalize}. What is exploited is the fact that
the $G$-action respects the graded structure on the polynomial ring
$K[V] = K[x_1 \upto x_n]$, so the algorithms actually extend to
degree-preserving actions on a graded algebra. For simplicity, we
assume in this section that $K$ is an algebraically closed field.

\begin{alg} \label{aTest} \mbox{}%
  \begin{description}
  \item[\bf Input:] A subalgebra $A = K[f_1 \upto f_m] \subseteq K[x_1
    \upto x_n] =: R$ generated by homogeneous polynomials, and a
    homogeneous, nonzero $a \in A$.
  \item[\bf Output:] ``true'' if $a^{-1} A \cap R = A$, ``false''
    otherwise.
  \end{description}
  \begin{enumerate}
    \renewcommand{\theenumi}{\arabic{enumi}}
  \item \label{aTest1} The first step is optional. Let $a_1 \upto a_s
    \in A$ such that~$a$ can be written as a product of powers of
    the~$a_i$. Run the algorithm with~$a$ replaced by~$a_i$. If the
    result is ``true'' for all~$i$, return ``true'' . Otherwise,
    return ``false''.
  \item Take new indeterminates $y_1 \upto y_m$ with $\deg(y_i) :=
    \deg(f_i)$, and form the ideal $\widehat{M} \subseteq K[x_1 \upto
    x_n,y_1 \upto y_m]$ generated by~$a$ and $y_i - f_i$ ($i = 1 \upto
    m$).
  \item Compute a finite homogeneous ideal basis $S$ of the
    elimination ideal $M := K[y_1 \upto y_m] \cap \widehat{J}$.
  \item \label{aTest4} For all $h \in S$, check whether $h(f_1 \upto
    f_m) \in A \cdot a$. (By homogeneity, this test comes down to
    testing the solvability of a system of linear equations.) If this
    is true for all $h \in S$, return ``true''. Otherwise, return
    ``false''.
  \end{enumerate}
\end{alg}

\begin{rem*}
  In MAGMA, each membership test in step~\ref{aTest4} can be done by a
  single call of the function {\tt HomogeneousModuleTest}.
\end{rem*}

\begin{proof}[Proof of correctness of \aref{aTest}]
  To show the correctness of step~\ref{aTest1}, let $a = a_1 a_2$
  with $a_i \in A$. Then it is clear that
  \[
  R \cdot a_1 \cap A = A \cdot a_1 \ \text{and} \ R \cdot a_2 \cap A =
  A \cdot a_2 \quad \Longleftrightarrow \quad R \cdot a \cap A = A
  \cdot a,
  \]
  which implies the correctness of step~\ref{aTest1}.

  The correctness of the remaining steps follows from the fact that
  the $h(f_1 \upto f_m)$ with $h \in S$ generate the ideal $R \cdot a
  \cap A \subseteq A$, which is easy to see.
\end{proof}

\begin{salg}[Unlocalizing the invariant ring of
  $G$-module] \label{aUnlocalize} \mbox{}%
  \begin{description}
  \item[\bf Input:] A linear algebraic group $G$ with a $G$-module
    $V$, a graded subalgebra $B \subseteq K[V]^G$ with a homogeneous,
    nonzero $a \in B$ such that $K[V]^G_a = B_a$.
  \item[\bf Output:] Homogeneous generators $f_1,f_2, \ldots$ of
    $K[V]^G$. If desired, the $f_i$ will form a minimal generating
    set. The procedure terminates after finitely many steps if and
    only if $K[V]^G$ is finitely generated.
  \end{description}
  \begin{enumerate}
    \renewcommand{\theenumi}{\arabic{enumi}}
  \item Set $m := 0$ and $A := K$. For $d = 1,2,3, \ldots$, perform
    steps~\ref{aUnlocalize2}--\ref{aUnlocalize5}.
  \item \label{aUnlocalize2} If $B \subseteq A$ and $a^{-1} A \cap
    K[V] = A$, return $f_1 \upto f_m$. The inclusion test comes down
    to testing the solvability of some systems of linear equations,
    and the second condition can be tested by \aref{aTest}.
  \item \label{aUnlocalize3} Compute a $K$-basis $S$ of $K[V]^G_d$,
    the space of homogeneous invariants of degree~$d$.
  \item \label{aUnlocalize4} This step is optional but recommended,
    and keeps the generating system minimal. Substitute $S$ by a set
    that is maximally linearly independent modulo $A_d$, the
    degree-$d$ part, of $A$.
  \item \label{aUnlocalize5} If $S = \{f_{m+1} \upto f_{m+r}\}$, set
    $A = K[f_1 \upto f_{m+r}]$ and $m := m+r$.
  \end{enumerate}
\end{salg}

\begin{rem*}
  For step~\ref{aUnlocalize3}, one can use
  \mycite[Algorithm~2.7]{kem.separating}, which is very
  efficient. Step~\ref{aUnlocalize4} comes down to the echelonization
  of a certain matrix with entries in $K$. In MAGMA, this step can be
  performed by a single call of the function {\tt
    HomogeneousModuleTestBasis}. The test of $B \subseteq A$ in
  step~\ref{aUnlocalize2} requires several calls of {\tt
    HomogeneousModuleTest}.
\end{rem*}

\begin{proof}[Proof of correctness of
  Semi-algorithm~\ref{aUnlocalize}]
  It suffices to show that for a subalgebra $A \subseteq K[V]^G$ the
  condition $A = K[V]^G$ is equivalent to $B \subseteq A$ and $a^{-1}
  A \cap K[V] = A$. Clearly $A = K[V]^G$ implies the other two
  conditions. Conversely, suppose that $B \subseteq A$ and $a^{-1} A
  \cap K[V] = A$, and let $f \in K[V]^G$. Since $K[V]^G_a = B_a$,
  there exists a positive integer~$k$ such that $f a^k \in B$ so $f
  a^k \in A$. By the second condition, this implies $f a^{k-1} \in A$,
  and then $f \in A$ by induction on~$k$.
\end{proof}

The following algorithm overcomes the limitation of \aref{aLocalize}
given by the condition that the invariant field has to be equal to the
field of fractions of the invariant ring. Because of
\pref{pHashimoto}, the algorithm generalizes to the situation where
$K[V]$ is replaced by a factorial domain $K[X]$. (In the case that the
group of units in $K[X]$ does not coincide with $K \setminus \{0\}$,
one should use the algorithm to compute $K[X]^{G^0}$ and then compute
$K[X]^G$ from this.)

\begin{alg}[Compute invariant rings of reductive
  groups] \label{aMaster} \mbox{}%
  \begin{description}
  \item[\bf Input:] A reductive group $G$ given as a subset of
    $K^m$ by a radical ideal $I_G \subseteq K[z_1 \upto z_m]$, and a
    $G$-module $V$ with the action given by
    \[
    \sigma \cdot v = \bigl(g_1(v,\sigma) \upto g_n(v,\sigma)\bigr)
    \]
    for $v \in V$ and $\sigma \in G$, where $g_i \in K[x_1 \upto
    x_n,z_1 \upto z_m]$.
  \item[\bf Output:] Homogeneous (and, if desired, minimal) generators
    of $K[V]^G$. If the algorithm terminates after finitely many
    steps, it also yields a correct result if $G$ is {\em not}
    reductive; but termination is only guaranteed if $G$ is reductive.
  \end{description}
  \begin{enumerate}
    \renewcommand{\theenumi}{\arabic{enumi}}
  \item Perform steps~\ref{aLocalize1} through~\ref{aLocalize7} of
    \aref{aLocalize}. So with $A \subseteq K(V)$ being the
    $K$-algebra generated by the fractions $f_j/h_j$, the inclusions
    \[
    K[V]^G \subseteq A \subseteq K(V)^G
    \]
    hold. Assume that $f_j$ and $h_j$ are coprime for all~$j$.
  \item \label{aMaster2} Let $a \in K[V]$ be the square-free part of
    the product of the~$h_j$. With $\mathcal{G}_G$ a Gr\"obner basis
    of $I_G$, set
    \[
    c := \frac{\NF_{\mathcal{G}_G}\bigl(a(g_1 \upto g_n)\bigr)}{a}.
    \]
    (It turns out that $c \in K[z_1 \upto z_m]$, and $\sigma^{-1}
    \cdot a = c(\sigma) a$ for $\sigma \in G$.)
  \item \label{aMaster3} With an additional indeterminate~$y$, let
    $\widehat{J} \subseteq K[y,z_1 \upto z_m]$ be the ideal generated
    by $I_G$ and $y - c$. Compute the elimination ideal $J := K[y]
    \cap \widehat{J}$.
  \item \label{aMaster4} If $J = (y^d - 1)$ with~$d$ a positive
    integer, then $a^d \in K[V]^G$. (The condition means that $G$ acts
    on~$a$ by multiplication by $d$th roots of unity.) So one can
    write
    \[
    \frac{f_j}{h_j} = \frac{g_j}{a^{e_j d}}
    \]
    with $e_j$ nonnegative integers and $g_j \in K[V]^G$. Then $B :=
    K[a^d,g_1 \upto g_k] \subseteq K[V]^G$ satisfies $B_{a^d} =
    K[V]^G_{a^d}$, so running Semi-algorithm~\ref{aUnlocalize} on $B$
    and $a^d$ yields (minimal) generators of $K[V]^G$.
  \item \label{aMaster5} This step is only reached when $J$ does not
    have the form $(y^d - 1)$. Let $H \subseteq G$ be the subgroup
    defined by $I_G$ and $c - 1$. By a recursive call to this
    algorithm, compute homogeneous generators of $K[V]^H$.
  \item \label{aMaster6} Use Algorithm~1.2 in
    \mycite{Derksen.Kemper06} to construct a $G$-module $W$ with a
    $G$-equivariant epimorphism $\map{\phi}{K[W]}{K[V]^H}$. It follows
    from the construction of $W$ that $H$ fixes $W$ point-wise. Since
    the quotient group $G/H$ is isomorphic to the multiplicative group
    $\Gm$, Derksen's algorithm (\mycite{Derksen:99}) can be used for
    computing $K[W]^G$. (When using the variant of Derksen's algorithm
    presented by \mycite{Kamke:Kemper:2011}, it is not necessary to
    make the action of $\Gm$ on $K[W]$ explicit.)  Applying~$\phi$ to
    the generators of $K[W]^G$ yields generators of $K[V]^G$. If they
    are not homogeneous, they can be substituted by their homogeneous
    parts. If desired, the generating set can be minimized.
  \end{enumerate}
\end{alg}

\begin{proof}[Proof of correctness of \aref{aMaster}]
  We first show that~$c$ from step~\ref{aMaster2} is a
  polynomial. With
  \[
  I := \bigr\{d \in K[V] \mid d \cdot f_j/h_j \in K[V] \ \text{for
    all} \ j\bigr\},
  \]
  the polynomial~$a$ generates the principal ideal $\sqrt{I}$. Since
  $I$ (and therefore also $\sqrt{I}$) is $G$-stable, this implies that
  for every $\sigma \in G$ there exists $c_\sigma \in K \setminus
  \{0\}$ such that $\sigma^{-1} \cdot a = c_\sigma a$. (This requires
  \pref{pHashimoto} if $K[V]$ is replaced by a factorial domain
  $K[X]$.) Choose a set $S \subseteq K[V]$ such that $S \cup \{a\}$ is
  a basis of $K[V]$ as a $K$-vector space. We can write
  \[
  g := \NF_{\mathcal{G}_G}\bigl(a(g_1 \upto g_n)\bigr) = c a +
  \sum_{i=1}^s c_i b_i
  \]
  with $c,c_i \in K[z_1 \upto z_m]$ and $b_i \in S$. For $\sigma \in
  G$ we have
  \[
  c(\sigma) a + \sum_{i=1}^s c_i(\sigma) b_i = g(\sigma) =
  a\bigl(g_1(\sigma) \upto g_n(\sigma)\bigr) = a\bigl(\sigma^{-1}
  \cdot x_1 \upto \sigma^{-1} \cdot x_n\bigr) = \sigma^{-1} \cdot a =
  c_\sigma a,
  \]
  where $g - a(g_1 \upto g_n) \in I_G$ was used for the second
  equality. This implies $c_i(\sigma) = 0$ for all~$i$, so $c_i \in
  I_G$. Since~$g$ is in normal form with respect to $\mathcal{G}_G$,
  the same is true for $c_i$, and we conclude $c_i = 0$. So, as
  claimed, $g/a = c \in K[z_1 \upto z_m]$, and the above equation
  implies $\sigma^{-1} \cdot a = c(\sigma) a$. It follows that
  $\mapl{\phi}{G}{\Gm}{\sigma}{c(\sigma)}$ is a homomorphism of
  algebraic groups. Since its image $\phi(G)$ is closed in $\Gm$ (see
  \mycite[Proposition~7.4B]{Humphreys}), it is either $\Gm$ or a cyclic
  group generated by some $d$th root of unity. The ideal $J$ from
  step~\ref{aMaster3} corresponds to the image closure of the map $G
  \xrightarrow{\phi} \Gm \hookrightarrow K$, so $J = \{0\}$ or $J =
  (y^d - 1)$. The second case is dealt with in step~\ref{aMaster4}.

  Now assume that $\phi(G) = \Gm$. Then $H$ from step~\ref{aMaster5}
  is the kernel of~$\phi$ and therefore a closed normal subgroup of
  $G$ with $G/H \cong \Gm$. It follows that the unipotent radical
  $R_u(H)$ is normal in $G$, so if $G$ is reductive, the same follows
  for $H$. Since $\dim(H) = \dim(G) - 1$, the termination and
  correctness of the recursive invocation of the algorithm in
  step~\ref{aMaster5} follows by induction. Because $H$ acts
  trivially on $K[V]^H$ and $W$, the $G$-action on both sets induces
  $\Gm$-actions on the categorical quotient $V \quo H$ and on $W$
  given by morphisms. (In fact, obvious though this sounds, an
  argument is required. For $W$, the assertion follows directly from
  \mycite[Corollary~6.10]{Borel}. For $V \quo H$, one can construct a
  homomorphism $K[V]^H \to K[\Gm] \otimes K[V]^H$ defining the
  $\Gm$-action by applying the same corollary to sub-$G$-modules $U
  \subseteq K[V]^H$ and then using the dual action $U \to K[\Gm]
  \otimes U$.) We obtain
  \[
  \phi\bigl(K[W]^G\bigr) = \phi\bigl(K[W]^{\Gm}\bigr) =
  \bigl(K[V]^H\bigr)^\Gm = K[V]^G,
  \]
  where the linear reductivity of $\Gm$ was used for the second
  equality. This proves the correctness of step~\ref{aMaster6}.
\end{proof}

It is interesting to compare \aref{aMaster} to other algorithms for
computing invariant rings of reductive groups. The algorithms that are
known to date are: Derksen's algorithm (see \mycite{Derksen:99}) for
linearly reductive groups, and the algorithm given by
\mycite{kem.separating} for reductive groups that are not linearly
reductive. Both these algorithms require the computation of a
(nonextended) Derksen ideal, and the cost of Derksen's algorithm is
dominated by this step. The algorithm from \mycite{kem.separating}
additionally requires the computation of another elimination ideal and
of a purely inseparable closure. This leads to further Gr\"obner basis
computations, whose cost rises significantly with the characteristic
of $K$. In comparison, \aref{aMaster} requires the computation of a
tamely extended Derksen ideal, which can be much cheaper than a
nonextended Derksen ideal since the number of indeterminates can be
reduced by the dimension of a typical orbit. In addition, Gr\"obner
basis computations are required when invoking \aref{aTest}. These can
significantly contribute to the cost of \aref{aMaster}. However, the
dependence on the characteristic of $K$ is limited to its influence on
the coefficient arithmetic.

This said, the only reasonable way to get a realistic idea of the
performance of the algorithms is to implement them and run
experiments. Derksen's algorithm is included in the standard
distribution of MAGMA. The algorithm from \mycite{kem.separating} and
\aref{aMaster} were implemented in MAGMA by the author.
In fact, the implementation of \aref{aMaster} requires the assumption
that $K(V)^G = \Quot\bigl(K[V]^G\bigr)$. By \tref{tItalian}, this is
guaranteed for the $G = \SL_2(K)$, a group that has kept invariant
theorists spellbound for more than a century. So it is natural to use
actions at $\SL_2(K)$ as benchmarks. The following table collects some
running times of the algorithms. The fist row shows the action. Here
$S^k(U)$ stands for the action on the $k$th symmetric power of the
natural module, i.e., the action on binary forms of degree~$k$, which
is an action of intense classical interest (see
\mycite[Example~2.1.2]{Derksen:Kemper}). The second row contains the
characteristic of $K$. The third and fourth rows show the running
times (in seconds) using the classical algorithm (i.e., Derksen's
algorithm for characteristic~$0$ and the one by
\mycite{kem.separating} for positive characteristic) and using
\aref{aMaster}. The computations were done on a 2.67 GHz Intel Xeon
X5650 processor. The symbols ``$> 600$'' stand for ``longer than the
author's patience.''

\bigskip

\begin{center}
  \begin{tabular}{|c||c|c|c|c|c|c|c|c|c|c|c|}
    \hline Action & \multicolumn{7}{|c|}{$V = S^4(U)$} &
    \multicolumn{4}{|c|}{$V = U \oplus S^3(U)$}  \\ \hline%
    Characteristic & 0 & 2 & 3 & 5 & 7 & 11 & 13 & 0 & 2 & 3 & 5 \\ \hline%
    Classical & 0.36 & 0.03 & 1.64 & 0.39 & 0.99 & 12.3 & 37 & 0.45 & 0.43 &
    3.63& $> 600$ \\ \hline%
    New & 0.35 & 0.1 & 0.32 & 0.21 & 0.26 & 0.22 & 0.23 & 0.74 & 0.1 &
    0.14 & 0.54 \\ \hline
  \end{tabular}
\end{center}

\bigskip

\begin{center}
  \begin{tabular}{|c||c|c|c|c|c|c|c|c|c|c|}
    \hline Action & \multicolumn{6}{|c|}{$V = S^2(U) \oplus S^2(U)$} &
    \multicolumn{4}{|c|}{$V = U \oplus S^2(U) \oplus S^2(U)$}  \\ \hline%
    Characteristic &  0 & 2 & 3 & 5 & 7 & 11 & 0 & 2 & 3 & 5 \\ \hline%
    Classical & 0.01 & 0.02 & 0.07 & 2.4 & 43 & $ > 600$ & 0.12 & 176
    & $> 600$ & $> 600$ \\ \hline%
    New &  0.05 & 0.03 & 0.02 & 0.05 & 0.08 & 0.04 & 0.21 & 0.22 &
    0.18 & 0.17 \\ \hline
  \end{tabular}
\end{center}


\bigskip

The ``outlier'' at $S^4(U)$ in characteristic~$3$ is probably linked
to the fact that the degrees of the generating invariants in
characteristic~$3$ differ from those in other characteristics. The
table shows that \aref{aMaster} is a good alternative to the
algorithms known to date. However, the main benefit is that it (and
the other algorithms from this paper) expand the scope of
computability beyond reductive groups. We finish the paper by an
example that was motivated by a question of Jonathan Elmer and Martin
Kohls.

\begin{ex}
  Consider the action of the upper unipotent group $U_3 \subseteq
  \GL_3(K)$ (with $\ch(K) = 0$) on the space $V = K^{3 \times 3}$ of
  $3 \times 3$ matrices by conjugation. It took 0.13 seconds to
  compute the invariant ring $K[V]^{U_3}$ by using \aref{aMaster}. It
  is a hypersurface of dimension~$6$. We describe a set of generating
  invariants. Let $M = (a_{i,j}) \in K^{3 \times 3}$ be a matrix,
  $(b_{i,j}) \in K^{3 \times 3}$ is its adjugate matrix, and $c_i$ the
  coefficients of the characteristic polynomial of $M$. Then the seven
  generating invariants (of degrees $1,2,3,1,2,3,3$) are given by
  mapping $M$ to
  \[
  c_1, \ c_2, \ c_3, \ a_{3,1}, \ b_{3,1}, \ a_{3,1} b_{3,2} - a_{3,2}
  b_{3,1}, \ a_{2,1} b_{3,1} - a_{3,1} b_{2,1}.
  \]
\end{ex}

\subsection*{Acknowledgements}

I am indebted to Evelyne Hubert for giving an earlier version of the
paper a thorough reading and suggesting, among other things, that I
should put in some more interesting applications. Only her remarks led
me to think of applying the ideas to actions of finite groups on
algebras over rings rather than fields. I am grateful to Mitsuyasu
Hashimoto for his patience in giving me explanations of his proof of
\pref{pHashimoto}. I would also like to thank Martin Lorenz and
Vladimir Popov for sharing their expertise, and Peter Fleischmann, Jim
Shank, and David Wehlau for interesting e-conversations.

\newcommand{\SortNoop}[1]{}

\end{document}